\documentclass[notitlepage]{article}
\usepackage[utf8]{inputenc}
\usepackage{amsmath}
\usepackage{titling}
\usepackage{amsthm}
\usepackage{geometry}
\usepackage{setspace}
\usepackage{enumitem}
\usepackage{amsmath}
\usepackage{amsfonts}
\usepackage{amssymb}
\usepackage{latexsym}
\usepackage{amsthm}
\usepackage{verbatim}
\usepackage{authblk}
\newtheorem{theorem*}{Theorem}
\newtheorem{corollary*}[theorem*]{Corollary}
\newtheorem{theorem}{Theorem}[section]
\newtheorem{corollary}[theorem]{Corollary}
\newtheorem{lemma}[theorem]{Lemma}
\newtheorem{proposition}[theorem]{Proposition}

\title{Finite Singular Orbit Modules for Algebraic Groups}
\author{Aluna Rizzoli}
\affil{Department of Mathematics, Imperial College, London\\ \texttt{a.rizzoli17@imperial.ac.uk}}
\begin{document}
\maketitle
\thispagestyle{empty}

\begin{abstract}
Building on the classification of modules for algebraic groups with finitely many orbits on subspaces \cite{finite}, we determine all faithful irreducible modules for simple and maximal-semisimple connected algebraic groups that are orthogonal and have finitely many orbits on singular $1$-spaces. This question is naturally connected with the problem of finding for which pairs of subgroups $H,K$ of an algebraic group $G$ there are finitely many $(H,K)$-double cosets. This paper provides a solution to the question when $K$ is a maximal parabolic subgroup $P_1$ of a classical group $SO_n$. We find an interesting range of new examples ranging from a $5$-dimensional module for $SL_2$ to the spin module for $B_6$ in characteristic $2$.
\end{abstract}

\section{Introduction}

In this paper we consider a question concerning double cosets in algebraic groups. Let $G$ be a simple algebraic group over an algebraically closed field $k$ of characteristic $p\geq 0$. The general problem consists of describing for which pairs of closed subgroups $H,K\leq G$ there are finitely many $(H,K)$-double cosets. This is a question that has attracted considerable interest thanks to the interesting range of examples coming from group theory and representation theory. A survey article on the problem has been written by Seitz \cite{SeitzDoubleCosets}.

If both $H$ and $K$ are parabolic subgroups, then by the Bruhat decomposition $G=\bigcup_{w\in W}B\dot{w}B$ we know that $|H\backslash G/K|<\infty$. A well-known result of Borel and Tits \cite[2.3]{BorelTits} implies that any closed connected subgroup of $G$ is either reductive, or lies in a parabolic subgroup of $G$.  We will only consider the case where one of the two subgroups, say $H$, is reductive. If $K$ is also connected reductive then a result of Brundan \cite[Thm. A]{brundan} says that if the number of double cosets is finite, then actually there is only $1$ double coset and the group $G$ has a factorization $G=HK$. Since such factorizations have been classified in \cite{factorizations}, we only consider the case when $K$ lies in a parabolic subgroup. The question remains open also in the case where $K$ is a maximal subgroup, and this is the setting that we will work with. 

When $G$ has type $A_n$ the problem has been settled thanks to \cite[Thm 3]{finite}. In \cite{finite} the authors determine all irreducible connected subgroups of $SL(V)$ with finitely many orbits on $k$-spaces in $V$. Since a maximal parabolic subgroup of $SL(V)$ is precisely the stabilizer $P_k$ of a $k$-space, we have that $H\leq SL(V)$ has finitely many orbits on $k$-spaces in $V$ if and only if there is a finite number of $(H,P_k)$-double cosets in $SL(V)$. A striking corollary of these results is that a simple algebraic group has finitely many orbits on the $1$-spaces of a rational irreducible module if and only if it has a dense orbit (see \cite[Cor. $1$]{finite}).

As pointed out in \cite{SeitzDoubleCosets} things are different when we consider other classical groups instead of $SL(V)$. For example, the group $H=G_2$ has infinitely many orbits on $1$-spaces on its $14$-dimensional Lie algebra $V=Lie(G_2)$ for $p\neq 3$. However it preserves a non-degenerate quadratic form and can therefore be regarded as a subgroup of $SO(V)$, and it is possible to see that it has finitely many orbits on singular $1$-spaces (see Proposition~\ref{B2 A2 G2}). In this paper we solve the general case of this problem. Our main result (Theorem~\ref{main theorem} later) classifies all orthogonal rational irreducible modules for simple connected algebraic groups with finitely many orbits on singular $1$-spaces. Theorem~\ref{main theorem 2} classifies all orthogonal rational irreducible modules for maximal semisimple connected algebraic groups with finitely many orbits on singular $1$-spaces. 

We call these modules \textit{finite singular orbit modules}. This is a natural first step in the classification of the modules with finitely many orbits on totally singular $k$-spaces and one that produces many interesting cases. We remark that we will be able to remove the maximality condition on semisimple groups once we have characterized all modules with finitely many orbits on totally singular $k$-spaces, which will be subject of forthcoming work.  

For an arbitrary simple connected algebraic group $G$, we denote by $T$ a maximal torus of $G$, by $B$ a fixed Borel subgroup of $G$, and let $\Phi$ be the root system of $G$ relative to $T$. By $T_i$ we indicate a torus of dimension $i$. Let $W=N_G(T)/T$ be the Weyl group. For $\alpha\in\Phi$, let $U_{\alpha}=\{U_{\alpha}(c):c\in k\}$ be the corresponding root subgroups of $G$. Let $\alpha_1,\dots,\alpha_n$ be a set of fundamental roots. For $1\leq i\leq n$ let $\lambda_i$ denote the corresponding fundamental dominant weight of $T$. If $\lambda$ is a dominant weight, let $V_G(\lambda)$ denote the irreducible rational $kG$-module of high weight $\lambda$. We use $P$ to denote a parabolic subgroup containing $B$ and by $P_k$ the maximal parabolic subgroup obtained by deleting the $k$-th node of the Dynkin diagram for $G$. When $G$ is a classical group and $P$ is maximal, $P$ corresponds to the stabilizer of a totally singular subspace. In a slight abuse of notation we are often going to refer to our groups by their type and to the modules by their highest weights, with for example $(SL_2,V_{SL_2}(4\lambda_1))$ being denoted by $(A_1,4\lambda_1)$. We will also sometimes refer to the notion of generic stabilizer, which in the case of the existence of finitely many orbits will just be the stabilizer of a point in the dense orbit. 
As in \cite{finite}, if $H$ is an irreducible subgroup of $SL(V)$ with finitely many orbits on $P_1(V)$, we say that $V$ is a \textit{finite orbit module} for $H$.

 Looking at \cite[Tables I, II]{finite}, in Table~\ref{tab:Orthogonal finite orbit modules} we list all the finite orbit modules that are orthogonal (excluding the natural modules for the orthogonal groups). This is done by determining which modules are self dual and with Frobenius-Schur indicator $"1"$, a process which is detailed in Section~\ref{list of candidates section}.
\begin{table}[h]
\centering
\ \begin{tabular}{||c c c||}
 \hline
 $H$ & $V$ & $p$\\ [0.5ex] 
 \hline\hline
 $A_1$ & $\lambda_1+p^i\lambda_1$ & $2$\\
 \hline
   $A_2 $ & $\lambda_1+\lambda_2$& $3$ \\
  \hline 
  $A_5$ & $\lambda_3$& $2$\\     
  \hline
 $B_3,B_4,B_5$&  $\lambda_3,\lambda_4,\lambda_5$& $2$ \\ 

  \hline

    $C_3 $ & $\lambda_2$& $3$\\
  \hline 
    $G_2 $ & $\lambda_1$& any \\
  \hline 
   $F_4 $ &  $\lambda_4$& $3$ \\
     \hline
   $E_7$&  $\lambda_7$& $2$\\ 
     \hline
  $A_1 B_n (p=2),A_1 C_n$&  $\lambda_1\otimes\lambda_1$&\\ 
  \hline
    $A_1B_3$ &  $\lambda_1\otimes \lambda_3$& $2$\\
    \hline  
      $A_1C_3$ & $\lambda_1\otimes \lambda_3$&$2$ \\
       \hline
      $A_1G_2$ & $\lambda_1\otimes \lambda_1$&$2$ \\

  \hline
 \end{tabular}
\caption{Orthogonal finite orbit modules \cite[Theorem $1$]{finite}}
\label{tab:Orthogonal finite orbit modules}
\end{table}

Here are our main results:

\begin{theorem*}\label{main theorem}
Let $H$ be a simple irreducible closed connected subgroup of $G=SO(V)$ such that $H$ has finitely many orbits on singular $1$-spaces. Then either $V$ is a finite orbit module for $H$ (see Table~\ref{tab:Orthogonal finite orbit modules}), or up to field or graph twists $(H,V)$ is as in Table~\ref{tab:Finite singular orbit modules for simple groups}.
\begin{table}[h]
\centering
 \begin{center}
 \begin{tabular}{||c c c c c c||} 
 \hline
 $H$ & $V$ & $\dim V$ & $p$ &Generic Stabilizer& Reference\\ [0.5ex] 
 \hline\hline
  $A_1$ & $4\lambda_1$ & $5$ & $\geq 5$ &$Alt_4$&\ref{sl2}\\     

  \hline
  $A_2$&  $\lambda_1+\lambda_2$&$8$  & $p \neq 3$ &$T_2 (p\neq 2)$, $T_2.3 (p=2)$&\ref{B2 A2 G2}\\
   $A_3$&  $\lambda_1+\lambda_3$&$14$  & $2$ &$T_3.Alt_4$&\ref{a3}\\ 
     \hline
  $B_2$&  $\lambda_2$&$10$  & $p \neq 2$ &$T_2.4$&\ref{B2 A2 G2}\\ 
   \hline
 $B_6$&  $\lambda_6$&$64$  & $2$&$P_1'P_1'<G_2G_2$&\ref{theorem B6 singular} \\ 
  \hline
  $G_2$&  $\lambda_2$&$14$  & $\neq 3$ & $T_2 (p\neq 2)$, $T_2.Dih_{6} (p=2)$&\ref{B2 A2 G2}\\ 
  \hline
    $C_3$&  $\lambda_2$&$14$  & $\neq 3$ &$(A_1)^3.(2^3.3) (p= 2)$ &\ref{c3}\\ 
     $C_4$&  $\lambda_2$&$26$  & $2$ &$(A_1)^4.(2^4.Alt_4)$ &\ref{c4d4}\\ 
  \hline
       $D_4$&  $\lambda_2$&$26$  & $2$ &$T_4.(2^3.Alt_4)$ &\ref{c4d4}\\ 
  \hline
   $F_4$ &  $\lambda_4$ & $26$ & $p\neq 3$ &$D_4.3$&\ref{F4} \\
   $F_4$ &  $\lambda_1$ & $26$ & $2$ &$D_4.3$&\\
     \hline
 \end{tabular}
\end{center}
\caption{Finite singular orbit modules for simple groups}
\label{tab:Finite singular orbit modules for simple groups}
\end{table}
\end{theorem*}

One particularly striking example that is not a finite orbit module is the group $SL_2(k)$ preserving a quadratic form on its $5$-dimensional irreducible representation $(p\geq 5)$. Here the number of orbits on $1$-spaces must be infinite simply by dimension considerations. We find however that there are only finitely many orbits on singular $1$-spaces, with a finite and disconnected stabilizer of a point in the dense orbit. Further examples when $H$ is a simple group arise from adjoint modules in small rank, the minimal module for $F_4$, the spin representation for $B_6$ and the alternating square of the natural module for $C_3$.

As mentioned before, it is the case that simple groups acting irreducibly have finitely many orbits on $1$-spaces if and only if they have a dense orbit. This turns out to still be true when looking at orbits on singular $1$-spaces.
We have the following corollary, analogously to \cite[Cor. 1]{finite}.

\begin{corollary*}\label{main corollary}
Let $H$ be a simple algebraic group over $k$ and let $V$ be a rational irreducible $kH$-module, with $H$ stabilizing a non-degenerate quadratic form on $V$. Then $H$ has finitely many orbits on singular $1$-spaces if and only if $H$ has a dense orbit on singular $1$-spaces.
\end{corollary*}

Clearly if a group has finitely many orbits on singular $1$-spaces, then one of these orbits must be a dense orbit. The other direction requires some work which is done on a case by case basis. We include a proof at the end of Section~\ref{conclusion of H simple}. 

When $H$ is maximal semisimple we have the following:
\begin{theorem*}\label{main theorem 2}
Let $H$ be a maximal semisimple irreducible closed connected subgroup of $G=SO(V)$ such that $H$ has finitely many orbits on singular $1$-spaces. Suppose that $H$ is not simple. Then either $V$ is a finite orbit module for $H$ (see Table~\ref{tab:Orthogonal finite orbit modules}), or up to field or graph twists $(H,V)$ is one of the following. 

\begin{table}[h]
\centering
 \begin{center}
 \begin{tabular}{||c c c c c c||} 
 \hline
 $H$ & $V$ & $\dim V$ & $p$ &Generic Stabilizer& Reference\\ [0.5ex] 
 \hline\hline
   $C_2C_2$ &  $\lambda_1\otimes \lambda_1$ & $16$ & $p\neq 2$ &$(A_1A_1).2$&\ref{sp4sp4}\\
     &   &  &  $p=2$&$U_3A_1$&\\
     $C_2C_n,n>2$ &  $\lambda_1\otimes \lambda_1$ & $8n$ & $p\neq 2$ &$(A_1A_1).2(Sp_{2n-4})$&\\
     &   &  &  $p=2$&$U_3A_1(Sp_{2n-4})$&\\
     \hline
 \end{tabular}
\end{center}
\caption{Finite singular orbit modules for maximal semisimple groups}
\label{tab:Finite singular orbit modules for maximal semisimple groups}
\end{table}
\end{theorem*}

The proofs of Theorem~\ref{main theorem} and Theorem~\ref{main theorem 2} will rely on the following steps.
Thanks to \cite{subgroupstructure} it is possible to classify the maximal closed connected subgroups of $G=SO(V)$ (see Prop.~\ref{subgroupstructure}). 
These belong to three different classes of subgroups, namely stabilizers of totally singular or non-degenerate subspaces; commuting products of classical groups stabilizing a tensor decomposition of $V$; or simple groups acting irreducibly and tensor indecomposably. 

Since the variety of singular $1$-spaces has dimension $\dim V-2$, we require $\dim V\leq \dim H+2$.
In order to produce a list of candidates for $H$, when $H$ is a simple group acting irreducibly on $V$, we use results in \cite{lubeck} about irreducible modules of small dimension for simple algebraic groups. If $p=2$ it is not immediately clear whether a given module is orthogonal, but we are able to deal with all the cases that we encounter using results in \cite{mikko}. Thus it is not difficult to obtain a list of candidates, but for a given candidate $(H,V)$ it can be difficult to determine whether $V$ is a finite singular orbit module. 
Our proofs will often give more information, such as the number of orbits, orbit representatives and stabilizers.

The layout of the paper is the following. We start by presenting some preliminary results in Section~\ref{preliminaries}. We then proceed to determine a list of candidates for simple groups acting irreducibly with finitely many orbits on singular $1$-spaces in Section~\ref{list of candidates section}, to then continue on a case by case basis. In Section~\ref{sl2 and f4} we deal with the group $A_1$ and its $5$-dimensional module, in Section~\ref{spin section} with the spin module for $B_6$, in Section~\ref{conclusion of H simple} we conclude the analysis of the cases where $H$ is simple. We finish with the analysis of semisimple cases in Section~\ref{tensor section}.

\section*{Acknowledgements}
The author would like to acknowledge his EPSRC funding and thank Professor Martin Liebeck for his invaluable support. 
\section{Preliminaries}\label{preliminaries}

In this section we present some results on linear algebra and algebraic groups that we are going to need in our proofs. We begin by recalling the structure of the maximal closed connected subgroups of classical groups. Let $V$ be a finite dimensional vector space over an algebraically closed field $k$ of characteristic $p$. We denote by $Cl(V)$ a classical group on $V$, i.e. one of $SL(V),Sp(V)$ and $SO(V)$.

\begin{theorem}\cite{subgroupstructure}\label{subgroupstructure}
Let $H$ be a closed connected subgroup of $G=Cl(V)$. Then one of the following holds:

\begin{enumerate}[label=(\roman*)]
\item $H\leq Stab_G(X)$ with $X\leq V$ a proper non-zero subspace which is either totally singular or non-degenerate, or $p=2$, $G=SO(V)$ and $X$ is non-singular of dimension $1$;
\item $V=V_1\otimes V_2$ and $H$ lies in a subgroup of the form $Cl(V_1).Cl(V_2)$ acting naturally on $V_1\otimes V_2$ with $\dim V_i\geq 2$ for $i=1,2$. When $G=SO(V)$ the only possibilities for $Cl(V_1).Cl(V_2)\leq G$ are $Sp_m\otimes Sp_n$ or $SO_m\otimes SO_n$. The orthogonal form on $V$ is given by the product of the bilinear forms on $V_1,V_2$, with rank-$1$ tensors being singular if $p=2$.
\item $H$ is a simple algebraic group acting irreducibly on $V$ and $V|_H$ is tensor indecomposable.
\end{enumerate}
\end{theorem}

It will sometimes be useful to be able to choose the field we are working with.

\begin{proposition}\cite[Prop. 1.1]{finite}\label{field independent}
Let $k\leq K$ be two algebraically closed fields of characteristic $p$. Let $G=G(K)$ be a connected reductive algebraic group over $K$. Denote by $G(k)$ the group of $k$-rational points of $G(K)$. Suppose that $G(K)$ acts algebraically on the affine variety $V(K)$, and the action is defined over $k$. Then $G(K)$ has finitely many orbits on $V(K)$ if and only if $G(k)$ has finitely many orbits on $V(k)$. If this holds the number of orbits is the same in each case.
\end{proposition}

We recall another general result from \cite{finite}. Assume $p>0$. For each power $q$ of $p$, let $\sigma_q$ be the Frobenius morphism of $SL(V)$, raising all matrix entries to the $q$th power relative to some fixed basis of $V$. Assume $G$ is a closed connected subgroup of $SL(V)$ which is $\sigma_q$-stable for some $q$. Let $G(q^e)$ denote the group of fixed points of $\sigma_{q^e}$ on $G$ and $V(q^e)$ denote the fixed points of $\sigma_{q^e}$ on $V$.
\begin{lemma}\cite[Lemma 2.10]{finite}\label{finiteOrbits}
Under the above assumptions $G\leq SL(V)$ has finitely many orbits on $P_k(V)$ if and only if there exists a constant $c$ such that $G(q^e)$ has at most $c$ orbits on $P_k(V(q^e))$ for all $e\geq 1$. In that case $G$ has at most $c$ orbits on $P_k(V)$
\end{lemma}

To conclude the section we recall the Lang-Steinberg Theorem. For an endomorphism of a simple algebraic group $G$, denote by $G_{\sigma}$ the fixed point subgroup. A \textit{Frobenius morphism} is an endomorphism of $G$ such that $G_{\sigma}$ is finite. We then have the following powerful theorem:
\begin{theorem}[Lang-Steinberg]
Let $H$ be a connected linear algebraic group, and suppose that $\sigma:H\rightarrow H$ is a surjective homomorphism of algebraic groups, such that $H_{\sigma}$ is finite. Then the map $h\rightarrow h^{-1}h^{\sigma}$ from $H\rightarrow H$ is surjective.
\end{theorem}

Note that in particular this holds for simple algebraic groups and Frobenius morphisms.

For an arbitrary group $G$ with an automorphism $\sigma$, we define $H^1(\sigma,G)$ to be the set of equivalence classes of $G$ under the equivalence relation $x\sim y\iff y=z^{-1}xz^{\sigma}$ for some $z\in G$. We then have the following proposition:

\begin{proposition}\label{lang-steinberg}\cite[I, $2.7$]{LangSteinberg}
Let $H,\sigma$ be as in the statement of the Lang-Steinberg Theorem. Suppose that $H$ acts transitively on a set $S$, and that $\sigma$ also acts on $S$ in such a way that $(sh)^{\sigma}=s^{\sigma}h^{\sigma}$ for all $s\in S,h\in H$. Then the following hold.
\begin{enumerate}[label=(\roman*)]
\item $S$ contains an element fixed by $\sigma$.
\item Fix $s_0\in S_{\sigma}$, and assume that $X=H_{s_0}$ is a closed subgroup of $H$. Then there is a bijective correspondence between the set of $H_{\sigma}$-orbits on $S_{\sigma}$ and the set $H^1(\sigma,X/X^0).$
\end{enumerate}
\end{proposition}

\section{List of simple irreducible candidates}\label{list of candidates section}
In this section we determine a list of suitable candidates for simple closed connected subgroups of $SO(V)$ with finitely many orbits on singular $1$-spaces. In particular we aim to prove the following theorem: 

\begin{theorem}\label{simple candidates}
Let $H<SO(V)$ be a connected simple algebraic group over $k$ acting irreducibly with finitely many orbits on singular $1$-spaces in $V$. Then either $V$ is a composition factor of the adjoint module for $H$ (as in Table~\ref{tab:Adjoint modules}), or $V$ is a finite orbit module for $H$, or $H$ is as in Table~\ref{tab:Simple candidates}.
\begin{table}[h]
 \begin{center}
 \begin{tabular}{||c c c c||} 
 \hline
 $H$  & $V$ & $\dim V$& $p$ \\ [0.5ex] 
 \hline\hline
  $A_1$  & $4\lambda_1$& $5$ & $\geq 5$ \\ 
 \hline
 $B_6$ & $\lambda_6$& $64$ & $2$ \\ 
  \hline
 $F_4$ & $ \lambda_4$ & $ 26$& $\neq 3$ \\ 
  $F_4$ & $ \lambda_1$& $ 26$ & $2$ \\ 
  \hline
  $C_n$, $n$ odd  & $ \lambda_2$& $ 2n^2-n-1$ & $p\nmid n$, $p\neq 2$ \\
  $C_n$, $n$ odd  & $ \lambda_2$& $ 2n^2-n-2$ & $p\mid n$, $p\neq 2$ \\
  $C_n$, $n\not\equiv 2 \ (\textrm{mod}\ 4)$  & $ \lambda_2$& $ 2n^2-n-gcd(n,2)$ & $p=2$ \\  
  \hline
  
 \end{tabular}
\end{center}
\caption{Simple candidates}
\label{tab:Simple candidates}
\end{table}
\end{theorem}

\begin{table}[bth]
 \begin{center}
 \begin{tabular}{||c c c c||} 
 \hline
 $H$ &  $V$ & $\dim V$ &$p$ \\ [0.5ex] 
 \hline\hline
  $A_n$           & $\lambda_1+\lambda_n$& $n^2+2n-1$  & $p\mid n+1$ \\ 
           & $\lambda_1+\lambda_n$& $n^2+2n$  & $p\nmid n+1$ \\ 
 \hline
 $B_n$   & $\lambda_2$ & $2n^2+n$& $\neq 2$ \\ 
 \hline
 $C_n$          & $2\lambda_1$ & $2n^2+n$& $\neq 2$ \\ 
  \hline
 $D_n$          &   $\lambda_2$& $2n^2-n$ & $\neq 2$ \\ 
 $D_n$        & $\lambda_2$& $2n^2-n-1$  & $p=2$, $n$ odd \\ 
 $D_n$         & $\lambda_2$ & $2n^2-n-2$& $p=2$, $n$ even \\ 
  \hline
 $G_2$          & $\lambda_2$   & $14$        & $\neq 3$ \\ 
  \hline
 $F_4$        & $\lambda_1 $   & $52$     & $\neq 2 $ \\ 
  \hline
 $E_6$          & $\lambda_2 $   & $ 77$       & $ 3$ \\ 
 $E_6$         & $\lambda_2 $  & $ 78$         & $ \neq 3$ \\ 
  \hline
 $E_7$       & $ \lambda_1$& $ 132$   & $2 $ \\ 
  $E_7$       & $ \lambda_1$ & $ 133$ & $\neq 2 $ \\ 
  \hline
 $E_8$& $\lambda_8 $& $248 $  & $any $ \\ 
 \hline
 \end{tabular}
\end{center} 
\caption{Adjoint modules \cite[1.10]{subgroupstructureExceptional}}
\label{tab:Adjoint modules}
\end{table}

Let $H\leq SO(V)$ be a simple connected algebraic group acting irreducibly on a finite singular orbit module $V$. Since the variety of singular $1$-spaces has dimension $\dim V-2$, we require $\dim V\leq \dim H+2$. Since $H\leq SO(V)$, $V$ must be an orthogonal module, i.e. self-dual and with Frobenius-Schur indicator $1$. 
In order to determine whether an irreducible highest weight module is self-dual we use the fact that $V(\lambda)^*\simeq V(-w_0(\lambda))$, where $w_0\in W$ is the longest element of the Weyl group (\cite[Prop. 16.1]{MT}).
In particular it is well known that the only cases when $w_0\neq -id$ correspond to root systems of type $A_n,D_n$ and $E_6$, when $-w_0$ induces a non-trivial graph automorphism of the Dynkin diagram.

To then determine the Frobenius-Schur indicator for $p\neq 2$ we use an observation in \cite[6.3]{lubeck}, based on two lemmas in \cite{steinberg}. 
\begin{lemma}\label{froebenius schur}
Let $G$ be a connected reductive algebraic group. If $V=V(\lambda)$ is a self-dual $G$-module and $p\neq 2$ then its Frobenius-Schur indicator is $+1$ if $Z(G)$ has no element of order $2$. Otherwise it is the sign of $\lambda(z)$ where $z$ is the only element of order $2$ in $Z(G)$, except for the case $G=D_l$ with even $l$ and $p\neq 2$, where $z$ is the element of $Z(D_l)$ such that $D_l/\langle z \rangle \simeq SO_{2l}$, with $D_l$ simply connected. This can be computed by \cite[6.2]{lubeck}.
\end{lemma}

We are now ready to prove Theorem~\ref{simple candidates}. 

\begin{proof}
Let $H\leq SO(V)$ and $V=V_H(\lambda)$ be a finite singular orbit module. If $\lambda$ is not $p$-restricted, then by Steinberg's Tensor Product Theorem \cite[Thm. 16.2]{MT} $\dim V$ is the product of the dimensions for non-zero restricted weights. However for all types except for $A_n$, the product of the dimensions of two irreducible modules $V_1,V_2$ violates the dimension bound (see the lists in \cite{lubeck}), i.e. $\dim V_1\cdot \dim V_2> \dim H+2$. For $A_n$ the weights $\lambda_1+p^i\lambda_1$ and $\lambda_1+p^i\lambda_n$ satisfy the dimension bound, but are already finite orbit modules. We can therefore assume that $\lambda$ is $p$-restricted. The list of possible modules satisfying the dimension bound are given by \cite[\S 6]{lubeck}

All we need to do is to go through the lists of $p$-restricted highest weight modules in\cite[\S 6]{lubeck} and determine which modules that satisfy the dimension bound are orthogonal. 

If $H=A_1$, then $\dim V\leq 5$ and if $\dim V<5$ we have a finite orbit module. Therefore $V=4\lambda_1$, which is clearly orthogonal.

 If $H=A_n,n\geq 2$, we get a bound of $\dim V\leq n^2+2n+2$. By  \cite[\S 6]{lubeck} the only modules satisfying the bound that are self-dual and not the adjoint module $V(\lambda_1+\lambda_n)$, are of weights $\lambda_{(n+1)/2}$, for $n=3,5$. These are finite orbit modules.
 
If $H=B_n$ we have the bound $\dim V\leq 2n^2+n+2$ and we find that the only non-adjoint self-dual modules satisfying the bound are $V(\lambda_n)$ for $n=3,4,5,6$ and $V(\lambda_1+\lambda_2)$ for $n=2,p=5$. These are all finite orbit modules apart from $V_{B_6}(\lambda_6)$ and $V_{B_2}(\lambda_1+\lambda_2)$ for $p=5$. 

We now show that $V_{B_6}(\lambda_6)$ is orthogonal only if $p=2$. Letting $z$ be the unique order $2$ element in the center of $B_6$ we see from \cite[6.2]{lubeck} that $\lambda_6(z)=-1$ and therefore by Lemma~\ref{froebenius schur} $V_{B_6}(\lambda_6)$ is not orthogonal if $p\neq 2$. By \cite[Thm 4.2]{mikko} $V_{B_6}(\lambda_6)$ is an orthogonal module if $p=2$. In the same fashion $V_{B_2}(\lambda_1+\lambda_2)$ is not orthogonal if $p=5$.

 If $H=C_n$ we have the bound $\dim V\leq 2n^2+n+2$. The only candidate module is therefore $V(\lambda_2)$. 

If $p\neq 2$, $V_{C_n}(\lambda_2)$ is orthogonal if and only if $n$ is odd (note that $n=2$ is a degenerate case which gives natural modules), again by \cite[6.2]{lubeck}. If $p=2$, $V_{C_n}(\lambda_2)$ is orthogonal if and only if $n\not\equiv 2 \ (\textrm{mod}\ 4)$, by \cite[Ex. 3.2]{mikko}.

If $H=D_n$ we have $\dim V \leq 2n^2-n+2$. There are no self-dual modules respecting the bound and that are neither adjoint or finite orbit modules.

For the exceptional types the only non-adjoint nor finite orbit modules satisfying the bound are the minimal module $V_{F_4}(\lambda_4)$ with $p\neq 3$ and the additional minimal module $V_{F_4}(\lambda_1)$ if $p=2$. These are both orthogonal (see \cite[Prop. 6.4]{mikko}).

This concludes our proof of Theorem~\ref{simple candidates}.

\end{proof}

\section{ The $5$-dimensional module for $SL_2$}\label{sl2 and f4}

One of the most interesting cases in Theorem~\ref{simple candidates} is given by the $5$-dimensional module $V(4\lambda_1)$ for $A_1$. In this section we prove the following:

\begin{theorem}\label{sl2}
For $p\geq 5$ the algebraic group $H=A_1$ has $3$ orbits on singular $1$-spaces in $V=V_{A_1}(4\lambda_1)$. The generic stabilizer is the finite group $Alt_4$.
\end{theorem}

By Proposition~\ref{field independent} we can assume that $k$ is the algebraically closed field $k=\overline{\mathbb{F}_p}$, with $p=char(k)\geq 5$ and $H=PGL_2(k)$, with underlying $2$-dimensional vector space $W$. We denote by $e,f$ a symplectic pair of vectors that forms a basis for $W$, and by $(\cdot,\cdot)$ the corresponding non-degenerate bilinear form on $W$. The irreducible $H$-module  with highest weight $4\lambda_1$ can be obtained by letting $H$ act naturally on $V:=S^4(W)$, the fourth symmetric power of $W$. 

Equip $V$ with the orthogonal form inherited from $W$, so that $(v_1\otimes v_2\otimes v_3 \otimes v_4,u_1\otimes u_2\otimes u_3 \otimes u_4)=(v_1,u_1) (v_2,u_2) (v_3,u_3) (v_4,u_4)$. Note that $H$ stabilises this form. 

Let $q=p^c$ and let $\sigma=\sigma_q$ be the standard Frobenius morphism acting naturally on both $H$ and $V$. It is well known that there exist subgroups $A\simeq Alt_4$, $S\simeq Sym_4$ with $$A\leq S\leq PGL_2(q)\leq H.$$ 

\begin{lemma}\label{1-space}
There exists a singular $1$-space $\langle v\rangle \in V$ such that $H_{\langle v \rangle }=A$.
\end{lemma}
\begin{proof}
The way that we obtain $A\leq H$ acting on $V$ is by starting with a group $SL_2(3)\leq SL_2(k)$ (which exists by \cite[Thm 6.26]{suzuki}). We know that $A=PSL_2(3)$ and we find its action on $V$ by letting $SL_2(3)$ act irreducibly on $W$ and then extending the action to $V$, since the center of $SL(2,3)$ is the kernel of the action of $SL(2,3)$ on $V$. Let $\chi_V$ be the irreducible character of $H$ on $V$.
By taking an irreducible $2$-dimensional character $\chi_2$ of $SL_2(3)$ we can compute the character $\chi_{V}\downarrow_A$ of $A$ using the fact, easily shown from first principles, that $$\chi_{V}\downarrow_A (g)=\chi(g^4)+\frac{\chi^4(g)-\chi^2(g^2)}{4}.$$ 

Hence $\chi_V\downarrow_A=\chi_1+\chi_1'+\chi_3$, where $\chi_1(1)=\chi_1'(1)=1, \chi_3(1)=3$ and $\chi_1,\chi_1'$ are non-trivial.

In particular $A$ fixes two $1$-spaces, which must be singular. We now want to show that $A$ is the full stabilizer of these $1$-spaces.
Let $\alpha$ be one of the $1$-spaces fixed by $A$. Suppose that $H_\alpha\neq A$. Then since $H=\bigcup_c PGL_2(p^c)$ and $H_\alpha\neq H$, $H_\alpha$ must intersect $PGL_2(q')$ in some subgroup $L$ such that $A< L < PGL_2(q')$, for some $q'=p^{c'}$. We now show that this is not possible.

Assume that $M$ is a minimal overgroup of $A$ in $PGL(2,q)$, so that $A< M\leq H_\alpha$. Then by \cite[Thm 6.25]{suzuki} $M$ can only be $S=Sym_4$ or $Alt_5$.
The group $Alt_5$ does not have any non-trivial $1$-dimensional characters, so it cannot be a subgroup of $H_{\alpha}$. Repeating the same calculations as for $A$, starting from a $2$-dimensional irreducible character for $GL(2,3)$, we find that $\chi_V\downarrow_S=\psi _2+\psi_2'+\psi_1$, where $\psi_2(1)=\psi_2'(1)=2$ and $\psi_1$ is trivial. Hence $S=Sym_4$ cannot be a subgroup of $H_{\alpha}$.

Therefore $A=Alt_4$ is the full stabilizer of $\alpha$, as required.
\end{proof}

The following lemma allows us to estimate the sizes of the stabilizers in $H_{\sigma}$. 

\begin{lemma}\label{correspondence}
Let $G,\sigma,S$ be as in the statement of Proposition~\ref{lang-steinberg}. Let $v\in S_{\sigma}$  and assume that $G_v$ is finite. Let $w$ be any element of $S_\sigma$ and let $g\in G$ be such that $w=vg$.

Then $$|Stab_{G_\sigma}(w)| \leq |Stab_{G_v} (gg^{-\sigma})|,$$ 
where on the left we have $G_\sigma$ acting on $S_\sigma$ and on the right we have the action of $G_v$ on itself by $x\rightarrow y^{-1}xy^\sigma$, which determines $H^1(\sigma, G_v)$.

\end{lemma}

\begin{proof}
It suffices to construct an injective map between the sets described. First note that since $G_v$ is finite, $G_v^0=1$ and $G_v$ is closed.
Therefore by Proposition~\ref{lang-steinberg} there is a bijection between the set of $G_{\sigma}$-orbits on $S_{\sigma}$ and $H^1(\sigma,G_v)$. 
This bijection can be described by taking an element $w\in S_{\sigma}$, writing it as $w=vg$ for some $g\in G$, and sending the orbit of $w$ to the class in $H^1(\sigma,G_v/G_v^0)$ with representative $gg^{-\sigma}$.
If we now take $h\in G_{\sigma}$ stabilizing $w=vg$, then $ghg^{-1} $ is an element of $G_v$ which fixes $gg^{-\sigma}$ (remember that the action is given by $x\rightarrow y^{-1}xy^\sigma$). So we have an injective map $\phi: Stab_{G_{\sigma}}(w)\rightarrow Stab_{G_v}(gg^{-\sigma})$ sending $h\rightarrow ghg^{-1}$. 
\end{proof}

What the above lemma tells us is that given any orbit with representative $v$ in $S_\sigma$, we can consider the stabilizer in the corresponding $H^1(\sigma, G_v)$ class and compare the size with the stabilizer of $v$ in $G_\sigma$. 

We begin by finding a complete list of orbits when passing to finite fields.

\begin{proposition}\label{sl2 orbits}
The group $PGL_2(q)$ acting on singular $1$-spaces in $V_\sigma = V(q)$ has $6$ orbits if $q \equiv 1 \ (\textrm{mod}\ 3)$ and $4$ otherwise. 
\end{proposition}

\begin{proof}
The strategy for this proof is based on the Lang-Steinberg Theorem. We have $PGL_2(q)=H_{\sigma}$ acting on $V_\sigma=V(q)$, for $V=S^4(\langle e,f \rangle)$.  We first note that $\langle e\otimes e\otimes e\otimes e\rangle$ has a Borel subgroup of size $q^2-q$ as stabilizer in $H_{\sigma}$. Next we observe that the $1$-space $\langle e\otimes e\otimes e\otimes f +e\otimes e\otimes f\otimes e+e\otimes f\otimes e\otimes e+f\otimes e\otimes e\otimes e\rangle$ is singular and has the subgroup of diagonal matrices as stabilizer, of size $q-1$.

Now by Lemma~\ref{1-space} we know that there is a singular vector $v\in V$ such that $H_{\langle v \rangle} =A=Alt_4$. Let $Z$ be the orbit of $\langle v\rangle$ in $P_1(V)$ under the action of $H$. 

We first claim that $Z$ is $\sigma$-stable. To do this it suffices to show that $\langle v\sigma\rangle\in Z$. There are two cases:
\begin{enumerate}[label=(\roman*)]
\item We have $\langle v\rangle=\langle v\sigma\rangle$, and therefore $Z$ is trivially $\sigma$-stable. This case corresponds to $q \equiv 1 \ (\textrm{mod}\ 3)$, since in order for a $1$-dimensional irreducible character of $SL(2,3)$ to exist over $\mathbb{F}_q$, $\mathbb{F}_q$ must contain a third root of unity (see proof of Lemma~\ref{1-space}).
\item We have $\langle v'\rangle:=\langle v\sigma\rangle$  being the other $1-$space stabilized by $A$. If we take any $s\in Sym_4\setminus A$ then $\langle v\rangle s \neq \langle v\rangle$ and since $A$ is normal in $S$, $\langle v\rangle s$ is stabilized by $A$ and we have $\langle v\rangle s=\langle v' \rangle$, showing that also in this case $Z$ is $\sigma$-stable. 
\end{enumerate} 

We are now ready to apply Proposition~\ref{lang-steinberg} to $H$ acting transitively on $Z$ with a compatible $\sigma$ action. We first see that there is $\langle v_0\rangle \in Z_\sigma=Z\cap P_1(V_\sigma)$. In case $(i)$, $\langle v \rangle \in Z_{\sigma}$ and we can assume that $v_0=v$. We know that $H_{\langle v_0\rangle}=A$, which is a finite subgroup of $H$, and therefore the orbits of $H_{\sigma}$ on singular $1$-spaces in $V_\sigma$ correspond to the equivalence classes $H^1(\sigma, A)$. Since $A$ is fixed by $\sigma$ these are just the $4$ conjugacy classes of $A$ with stabilizers of size $3,3,4,12$. 

In case $(ii)$ we know by Lang-Steinberg that $s=h^{-1}h^{\sigma}$ for some $h\in H$, therefore $\langle v\rangle s=\langle v' \rangle$ implies that $\langle vh^{-1} \rangle \in Z_{\sigma}$. We can therefore assume that $v_0=vh^{-1}$.
Here we get $H_{\langle v_0\rangle}=hAh^{-1}$ and the orbits of $H_{\sigma}$ on singular $1-$spaces in $V_\sigma$ correspond to the equivalence classes $H^1(\sigma, hAh^{-1})$. We compute the sizes of the stabilizers. Let $x:=hah^{-1}$ be an element of $hAh^{-1}$. Let $y:=ha_1h^{-1}$ be an element of  $hAh^{-1}$ fixing $hah^{-1}$, i.e. an element such that $x^{-1}yx^{\sigma}=y$. Substituting in $x,y$ we get $$ha_1^{-1}h^{-1} hah^{-1}(ha_1h^{-1})^{\sigma}=hah^{-1},$$ which is equivalent to
$$ha_1^{-1}asa_1h^{-\sigma}=hah^{-1},$$ since $a_1$ is fixed by $\sigma$ and $h^{-1}h^{\sigma}=s$.
This is equivalent to 
$ a_1^{-1} as a_1=as.$

Now $as\in Sym_4\setminus Alt_4$ and the centralizer in $Alt_4$ of any element of $Sym_4\setminus Alt_4$ has size $2$. There are therefore two classes in $H^1(\sigma, hAh^{-1})$, with stabilizers of size $2$.

By Lemma~\ref{correspondence} we then have orbits with stabilizers of size $q^2-q, q-1,n_1,n_2,n_3,n_4$, with $n_1\leq 3, n_2\leq 3, n_3\leq 4, n_4\leq 12$, for case $(i)$ and $q^2-q, q-1,m_1,m_2$, with $m_1,m_2\leq 2$  for case $(ii)$.
Note that we do not yet know that these are all the orbits. Observe that $|PGL_2(q)|=q(q-1)(q+1)$ and there are $1+q+q^2+q^3$ singular $1$-spaces. If we add up the sizes of the orbits assuming that the above inequalities between the sizes of the stabilizers are equalities, we find that we already get $1+q+q^2+q^3$ points. Therefore the orbits listed form a complete set of orbits for $PGL_2(q)$ acting on $V(q)$, with stabilizers for case $(i)$ of sizes $q^2-q, q-1,3,3,4,12$, and of sizes $ q^2-q, q-1,2,2$ for case $(ii)$.
\end{proof}

We can finally conclude the proof of Theorem~\ref{sl2}.
The proof of Proposition~\ref{sl2 orbits} shows that in the algebraic case $H$ has $3$ orbits on singular $1$-spaces in $V$, with stabilizers $B,T_1,Alt_4$. When passing to finite fields, the orbit with stabilizer $Alt_4$ splits into $4$ orbits if $q \equiv 1 \ (\textrm{mod}\ 3)$ and in $2$ orbits otherwise.

\section{The 64-dimensional spin module for $B_6$}\label{spin section}

In this section we deal with the case where $H=B_6$ and $X=V(\lambda_6)$ is the spin module for $B_6$, in characteristic $2$. In particular we prove the following:

\begin{theorem}\label{theorem B6 singular}
The group $H=B_6$ has finitely many orbits on singular $1$-spaces in $X=V_{B_6}(\lambda_6)$ over an algebraically closed field of characteristic $2$.
\end{theorem}

Let us recall the main result in  \cite{popov}. The adopted notation is defined below.

\begin{theorem}\label{stabs in char 0}
Let $G=D_7$ over an algebraically closed field of characteristic $0$. Every non-zero vector in $V_{D_7}(\lambda_6)$ is in the $G$-orbit with representative one the vectors given in the following table.

\begin{center}
 \begin{tabular}{||c c c c ||} 
 \hline
 Orbit type & Spinor $x$ & $G_x^0$ & $[G_x:G_x^0]$ \\ [0.5ex] 
 \hline\hline
  $2$ & $ 1 $ & $U_{21} .SL_7$ & $1$ \\ 
 \hline
 $3$ &$ 1+e_4e_5e_6e_7$& $U_{27}.(SL_3\times Spin_7)$ & $1$  \\
 \hline
 $4$&$ 1+e_1e_2e_3e_4e_5e_6$& $U_{12}. SL_6$ & $2$  \\
 \hline
$5$ &$ 1+e_1e_2e_3e_7+e_1e_2e_3e_4e_5e_6$& $U_{21} .(SL_3\times SL_3) $ & $2$  \\
 \hline
 $6$ &$\lambda(1+e_1e_2e_3e_7+e_4e_5e_6e_7+e_1e_2e_3e_4e_5e_6):\lambda\in k^*$& $G_2\times G_2$ & $2$  \\ 
 \hline
 $7$ &$1+e_1e_2e_3e_7+e_3e_4e_5e_7+e_1e_2e_3e_4e_5e_6$& $U_{19} .(SL_2\times_{\mathbf{Z}_2}Sp_4) $ & $1$  \\ 
 \hline
 $8$ &$ 1+e_1e_2e_3e_7+e_3e_4e_5e_7+e_2e_5e_6e_7+e_1e_2e_3e_4e_5e_6$&$U_{14}. G_2$& $1$  \\ 
 \hline
 $9$ &$ 1+e_2e_3e_5e_6+e_1e_3e_4e_6$& $U_{26}. (Sp_6\times_{\mathbf{Z}_2}T_1)$ & $1$  \\ 
 \hline
  $10$&$1+e_1e_2e_3e_7+e_2e_3e_5e_6+e_1e_3e_4e_6$& $U_{26.}SL_4$ & $1$  \\ 
 \hline
\end{tabular}
\end{center}

\end{theorem}

We divide the proof of Theorem~\ref{theorem B6 singular} into two parts. In characteristic $0$ Igusa (\cite{igusa}) determined complete lists of representatives and stabilizers for the orbits on vectors in spin modules up to $D_6$, while Popov (\cite{popov}) did the same for $V_{D_7}(\lambda_6)$. In \cite{Gatti} we have a list of orbit representatives for $V_{B_6}(\lambda_6)$ if $p=0$. We are able to use arguments in these papers to show that the list of orbit representatives for $D_7$, given in Theorem~\ref{stabs in char 0}, is in fact a list of representatives in arbitrary characteristic. We then analyze how the $D_7$-orbits split when restricting to $H=B_6$.                      

To do all of this we are going to need some notation, which we provide over arbitrary characteristic. We use a basis $\{e_1,\dots e_7,e_{8},\dots e_{14}\}$ for the natural module $V=V_{14}$ with quadratic form $Q$ for $D=D_7$, with $\{e_i,e_{7+i}\}$ being hyperbolic pairs for $i\leq 7$. We will sometimes use $f_1,\dots ,f_7$ to denote the vectors $e_8,\dots , e_{14}$, so that a set of hyperbolic pairs is given by $\{e_i,f_i\}$. Let $L,M$ be the totally singular subspaces $\langle e_1,\dots,e_7\rangle $ and $\langle e_8,\dots,e_{14}\rangle$ respectively.
We denote by $C$ the Clifford algebra of $(V,Q)$, and by $x\rightarrow x'$ its canonical antiautomorphism, i.e. the automorphism that reverses the order in all products. Then $C=C^+\oplus C^-$, where $C^+$ is the space of even elements  and $C^-$ is the space of odd elements. The Clifford group is $G^*=\{s\in C|s$ is invertible in $C$ and $sVs^{-1}=V\}$. The even Clifford group is $(G^*)^+=G^*\cap C^+$. The spin group $D_7$ is $\{s\in (G^*)^+|ss'=1\}$. The natural representation of $D_7$ corresponds to the restriction to $D_7$ of the vector representation $\Theta: G^*\rightarrow Aut(V,Q)$ given by $\Theta(s)\cdot v=svs^{-1}$. 

Put $e_L=e_1e_2\dots e_7$ and $e_M=e_8e_9\dots e_{14}$. We denote by $C_W$ the subalgebra of $C$ generated by the elements of a subspace $W\subset V_{14}$. Then $Ce_M$ is a minimal left ideal in $C$, and the correspondence $x\rightarrow xe_M$ generates an isomorphism $C_L\rightarrow Ce_M$ of vector spaces. So for any $s\in C,x\in C_L$ there exists a unique element $y\in C_L$ for which $sxe_M=ye_M$. Setting $\rho(s)\cdot x = s\cdot x=y$ gives us the spinor representation $\rho$ of the algebra $C$ in $C_L$. Let $X=C_L\cap C^+$. Then restricting $\rho$ to $Spin_{14}$, we get the half-spinor representation of $D_7$ in $X$.

A maximal torus $T$ of $D_7$ is generated by elements $s_i(\lambda):=\lambda^{-1}+(\lambda-\lambda^{-1})e_ie_{7+i}$, where $\lambda\in k^*$ and $1\leq i\leq 7$. The $1$-dimensional root subgroups are parametrised by $s_{i,j}(\lambda):=1+\lambda e_ie_j$ where $\lambda\in k$ and $1\leq i,j\leq 14$, with $(e_i,e_j)=0$. 

A product of the form $e_{i_1}e_{i_2}\dots e_{i_p}$ is a monomial of degree $p$. The spinor $1$ is considered of degree $0$. We denote by $C_p$ the subspace of homogeneous elements of degree $p$ in $C$. The subspace of homogeneous elements of degree $p$ in $X$ is $X_p=C_p\cap X$. If $x\in C$, then denote by $x_p$ the homogeneous component of degree $p$ of the element $x$. If $W$ is a totally singular subspace of $V$, then we identify $C_W$ with the exterior algebra of $W$.

Let $W$ be a totally singular subspace of $V$ with basis $\{w_1,\dots , w_r\}$. If $u=\sum_{i<j}a_{ij}w_iw_j$ is an element of $(C_W)_2$, set $\exp u=\prod_{i<j}(1+a_{ij}w_iw_j)$.  Then $\exp (u+u')=\exp u \exp u'$, $(\exp u)^{-1}=\exp (-u)$.

We denote by $L_0$ the space spanned by $e_1,\dots,e_6$, by $M_0$ the space spanned by $e_8\dots e_{13}$, by $V_0$ the space $L_0\oplus M_0$, and by $C_0$ the algebra $C_{V_0}$. Let $e_{L_0}=e_1\dots e_6$, $e_{M_0}=e_8\dots e_{13}$ and let $e^*_{i_1,\dots,i_k}$ be the complement monomial to $e_{i_1,\dots,i_k}$ in $L_0$, i.e. the monomial in $e_1,\dots  ,e_6$ such that $e_{L_0}=e^*_{i_1,\dots,i_k} e_{i_1}\dots e_{i_k}$. 

We conclude with a description of the quadratic form stabilised by $H=B_6=(D_7)_{e_7+f_7}$ on $X$ when $p=2$. A set of hyperbolic pairs is given by $(e_{i_1,\dots,i_k},e^*_{i_1,\dots,i_k})$ in $L_0$ when $k$ is even, combined with $(e_{i_1,\dots,i_k}e_7,e^*_{i_1,\dots,i_k}e_7)$ when $k$ is odd.

\subsection{The orbit representatives}
Our aim is to show that the list of representatives in Theorem~\ref{stabs in char 0} is a complete list of representatives for any characteristic. We will prove the following:

\begin{proposition}\label{list of representatives}
Over an arbitrary algebraically closed field $k$, a nonzero spinor in $X=V_{D_7}(\lambda_6)$ is $Spin_{14}$-equivalent to one of the following spinors:
\begin{center}
 \begin{tabular}{||c c  ||} 
 \hline
 Orbit type & Spinor \\ [0.5ex] 
 \hline\hline
  $2$ & $ 1 $  \\ 
 \hline
 $3$ &$ 1+e_4e_5e_6e_7$  \\
 \hline
 $4$&$ 1+e_1e_2e_3e_4e_5e_6$  \\
 \hline
$5$ &$ 1+e_1e_2e_3e_7+e_1e_2e_3e_4e_5e_6$ \\
 \hline
 $6$ &$\lambda(1+e_1e_2e_3e_7+e_4e_5e_6e_7+e_1e_2e_3e_4e_5e_6):\lambda\in k^*$  \\ 
 \hline
 $7$ &$1+e_1e_2e_3e_7+e_3e_4e_5e_7+e_1e_2e_3e_4e_5e_6$  \\ 
 \hline
 $8$ &$ 1+e_1e_2e_3e_7+e_3e_4e_5e_7+e_2e_5e_6e_7+e_1e_2e_3e_4e_5e_6$  \\ 
 \hline
 $9$ &$ 1+e_2e_3e_5e_6+e_1e_3e_4e_6$ \\ 
 \hline
  $10$&$1+e_1e_2e_3e_7+e_2e_3e_5e_6+e_1e_3e_4e_6$  \\ 
 \hline
\end{tabular}
\end{center}

\end{proposition}

The strategy will consist of arguing that the methods used by Popov in \cite{popov} and Igusa in \cite{igusa} when $p=0$, can also be applied in positive characteristic. 

The following is a special case of \cite[Lemma~$1$]{igusa}. The proof still holds in arbitrary characteristic.

\begin{lemma}\label{dropping the 2}
Every element $x\neq 0$ of $X$  is $D_7$-equivalent to a spinor in  $1+X_4+X_6$.
\end{lemma}

Let us therefore start with a spinor $x=1+x_4+x_6$. There are two possibilities, either $x_6=0$ or $x_6\neq 0$. We first deal with the second case: $x_6\neq 0$. As noted in \cite[\S $2$]{popov}, by applying to $x$ an appropriate transformation of $SL_7\leq Spin_{14}$, we can arrange that $x_6=e_{L_0}=e_1e_2e_3e_4e_5e_6$.

Write $x_4=ye_7+z$ for $y\in (C_{L_0})_3$ and $z\in (C_{L_0})_4$. We then have $$x=1+ye_7+z+e_1e_2e_3e_4e_5e_6.$$ Consider the Levi subgroup $GL_7=(D_7)_{L}$. As shown in \cite{popov}, by acting on $x$ with the group $GL_6=(GL_7)_{e_7,f_7}$, it is possible to bring the spinor $x$ into one of the following types:

\begin{enumerate}[label=(\arabic*),start=4]
\item $1+ye_7+e_{L_0};$
\item $1+ye_7+e^*_{1,4}+ e_{L_0};$
\item $1+ye_7+e^*_{1,4}+e^*_{2,5}+ e_{L_0};$
\item $1+ye_7+e^*_{1,4}+e^*_{2,5}+e^*_{3,6}+ t e_{L_0}$ with $t\in k^*$.
\end{enumerate}

Note that we have maintained the same numberings of the types as in \cite{popov} for an easier comparison.
By using just elements of the root subgroups, exponentials and $SL_6\leq Spin_{14}$, it is easy to see that types $(4),(5),(6)$ in the above list are equivalent (for the details see \cite[\S 2]{popov}). The remaining case is slightly more involved.

\begin{lemma}
If $t\neq \pm 2$, then a spinor of the form $1+ye_7+e^*_{1,4}+e^*_{2,5}+e^*_{3,6}+ t e_{L_0}$ with $t\in k^*$ is $D_7$-equivalent to a spinor of type $(4)$, i.e. of the form $1+ye_7+e_{L_0}$.
\end{lemma}

\begin{proof}
The first part of the proof of Lemma~$4$ in \cite{popov} shows that it is possible to reduce a spinor of type $(7)$ to a spinor of type $(4)$, given that we can find $\lambda, \mu \in k$ such that $\lambda^2-\lambda t+1=\mu^2-\mu t+1=0$, $\lambda \mu \neq 1$ and $t-\lambda-\mu\neq 0$. Now if $p=2$ the quadratic polynomial $w^2-t w+1$ has two distinct roots, since $t\neq 0$ by assumption on the type. The same is true over any other characteristic as by assumption $t\neq \pm 2$. Therefore we can set $\mu = \lambda$ to be one of the roots and we will satisfy the conditions required.
\end{proof}

This gives the following corollary if $p=2$.

\begin{corollary}\label{x6 not 0 in char 2}
If $p=2$ the spinors of type $(4),(5),(6),(7)$ are $D_7$-equivalent.
\end{corollary}

We now study the spinors of type $(4)$. Continuing on the lines of Popov we have the following:

\begin{lemma}\label{type 4 spinors}
Every spinor of type $(4)$ with $y\neq 0$ is $D_7$-equivalent to one of the following spinors:
\begin{enumerate}[label=(\arabic*),start=9]
\item $1+e_1e_2e_3e_7+e_{L_0};$
\item $1+e_1e_2e_3e_7+e_4e_5e_6e_7+ te_{L_0}, t\in k^*;$
\item $1+e_1e_2e_3e_7+e_3e_4e_5e_7+ e_{L_0};$
\item $1+e_1e_2e_3e_7+e_4e_3e_5e_7+e_6e_5e_2e_7+e_{L_0}$.
\end{enumerate}
\end{lemma}

\begin{proof}
The proof is the same as \cite[Lemma 5]{popov} with one adjustment. Lemma $5$ in \cite{popov} makes use of the classification of trivectors of six-dimensional space given in \cite{reichel}. This is valid only for an algebraically closed field of characteristic $0$. A classification regardless of the structure of the underlying field was obtained in \cite{revoy}. For an algebraically closed field of arbitrary characteristic it turns out that the spinor representatives are the same as in characteristic $0$ (see \cite[Prop. $1.2$]{tri6}).
\end{proof}

In order to deal with spinors of type $(7)$ when $t=\pm 2$ ($p\neq 2$), Popov classifies the trivectors of six-dimensional space with respect to $Sp_6$. Since the classification provided in \cite[Theorem $1$]{popov} holds when $p\neq 2$,  by \cite[\S $4$]{popov} we have the following:

\begin{proposition}
If $p \neq 2$, every spinor of type $(7)$ with $t=\pm 2$ is $D_7$-equivalent to a spinor of type $(9)-(12)$ or to one of the following spinors:
 \begin{enumerate}[label=(\arabic*),start=21]
\item $1+e_{1,4}^*+e_{2,5}^*$;
\item $1+e_1e_2e_3e_7+e_{1,4}^*+e_{2,5}^*$.
\end{enumerate}

\end{proposition}

This together with Corollary~\ref{x6 not 0 in char 2} proves the following:

\begin{proposition}\label{x6 not 0}
Every spinor of the form $1+x_4+x_6$ with $x_6\neq 0$ is $D_7$-equivalent to one of the following spinors:
\begin{enumerate}[label=(\arabic*),start=16]
\item $1+e_{L_0};$
\item $1+e_1e_2e_3e_7+e_{L_0};$
\item $t(1+e_1e_2e_3e_7+e_4e_5e_6e_7+ e_{L_0}), t\in k^*;$
\item $1+e_1e_2e_3e_7+e_3e_4e_5e_7+ e_{L_0};$
\item $1+e_1e_2e_3e_7+e_4e_3e_5e_7+e_6e_5e_2e_7+e_{L_0}$;
\item $1+e_{1,4}^*+e_{2,5}^*$;
\item $1+e_1e_2e_3e_7+e_{1,4}^*+e_{2,5}^*$.
\end{enumerate}
\end{proposition}

We are finally left with the case $x_6=0$, where $x=1+x_4$. Continuing to follow Popov we have the following:

\begin{lemma}
Every spinor of the form $1+x_4$ is $D_7$-equivalent to one of the following spinors:
\begin{enumerate}[label=(\arabic*),start=25]
\item $1;$
\item $1+e_4e_5e_6e_7;$
\item $1+e_4e_5e_6e_7+e_2e_1e_6e_7;$
\item $1+e_4e_5e_6e_7+e_2e_1e_6e_7+e_1e_3e_4e_7;$
\item $1+e_4e_5e_6e_7+e_2e_1e_3e_7$;
\item $1+e_4e_5e_6e_7+e_1e_2e_6e_7+e_1e_2e_4e_5$;
\item $1+e_4e_5e_6e_7+e_1e_2e_6e_7+e_1e_2e_4e_5+e_1e_3e_4e_6$;
\item $1+e_4e_5e_6e_7+e_2e_1e_3e_7+e_2e_3e_5e_6$;
\item $1+e_4e_5e_6e_7+e_2e_1e_3e_7+e_2e_3e_5e_6+e_1e_3e_4e_6$;
\item $1+t(e_4e_5e_6e_7+e_2e_1e_3e_7+e_2e_3e_5e_6+e_1e_3e_4e_6+e_1e_2e_4e_5), t\in k^*$.
\end{enumerate}
\end{lemma}

\begin{proof}
The proof is the same as \cite[Lemma 9]{popov} with one adjustment. The classification of trivectors of seven-dimensional space found in \cite{schouten} is used, which is only valid if $p=0$. However by \cite[Theorem 2.1]{tri7}, the same classification still holds over an algebraically closed field of arbitrary characteristic.
\end{proof}

To conclude the proof of Proposition~\ref{list of representatives} it remains to show that all types between $(25)$ and $(34)$ are equivalent to a type between $(17)$ and $(26)$, since types $(17),(18),(19),(20),(21),(22),(23),(25),(26)$ form the list in the statement of Proposition~\ref{list of representatives}.

This is dealt with in \cite{popov} with a series of lemmas proving equivalences (Lemmas $10$ to $17$ in \cite{popov}). The arguments only use root elements and exponentials. The only point where things are not the same when $p=2$ can be found in Lemma $16$ and Lemma $17$. In Lemma $16$
we can however just apply $s_{10,13}(1)s_{2,5}(1)$ to conclude if $p=2$. The same Lemma actually shows that types $(33)$ and $(34)$ are equivalent, so Lemma $17$ is dealt with as well. This completes the proof of Proposition~\ref{list of representatives}.

Note that we have not shown that the list in Proposition~\ref{list of representatives} is a set of orbit representatives, since we would need to prove that the spinors are pairwise inequivalent. For the purpose of working with the orbits on singular $1$-spaces we will not need to achieve a complete classification of the orbits, as discussed in the following subsection.

\subsection{Restriction to $Spin_{13}$}

In this subsection we conclude the proof of Theorem~\ref{theorem B6 singular}.  Recall that in our setting we have $H=B_6=(D_7)_{e_7+f_7}$ over an algebraically closed field $k$ with $p=2$, and the spin module $X=V_{D_7}(\lambda_6)$. To prove Theorem~\ref{theorem B6 singular} we need to show that $X\downarrow B_6=V_{B_6}(\lambda_6)$ is a finite singular orbit module. We start by outlining the adopted strategy.
Recall that by Proposition~\ref{list of representatives}  there are at most $9$ orbits for $D_7$ acting on $1$-spaces in $X$. The following elementary lemma describes how each one of these orbits splits when restricting to $B_6$.

\begin{lemma}\label{orbit correspondence D7 B6}
Let $\Delta=\langle x \rangle ^{D_7}$ be an orbit of $D_7$ on $P_1(X)$, with $S=(D_7)_{\langle x \rangle}$. Then there is a bijective correspondence between the orbits of $H$ on $\Delta$ and the orbits of $S$ on non-singular $1$-spaces in $V=V_{14}$. More specifically if $g\in D_7$ and $\alpha=\langle e_7+f_7\rangle g$ is a non-singular $1$-space, then the orbit $\alpha ^{S}$ corresponds to the orbit $\langle g\cdot x \rangle^{B_6}$.
\end{lemma}

\begin{proof}
The elements of the orbit $\Delta$ are in bijective correspondence with the coset space $[D_7:S]$ and therefore the orbits of $H$ on $\Delta$ are in bijective correspondence with the $(H,S)$-double cosets in $D_7$. Since $H=(D_7)_{e_7+f_7}$ the first statement is proven. More precisely the element $g\cdot \langle x\rangle$ corresponds to the coset $gS$ and in turn the orbit $\langle g\cdot x \rangle ^{H}$ corresponds to the double coset $HgS$ which we can identify with the orbit of $S$ containing $\alpha = \langle e_7+f_7\rangle g$.
\end{proof}

We will be able to prove Theorem~\ref{theorem B6 singular} without knowing precisely what the stabilizer $S$ of $\alpha$ is in $D_7$, thanks to the following lemma:

\begin{lemma}\label{subgroup stab is enough}
Let $\Delta=\langle x \rangle ^{D_7}$ be an orbit of $D_7$ on $P_1(X)$, with $S=(D_7)_{\langle x \rangle}$. Let $S'\leq S$. Let $\Lambda \leq \Delta$ be the set of singular $1$-spaces in $\Delta$. Let $A=\{ \langle e_7+f_7\rangle g : Q (g\cdot x)=0,g \in D_7\}$. If the number of $S'$-orbits on $A$ is finite, then so is the number of $H$-orbits on $\Lambda$.
\end{lemma}

\begin{proof}
By Lemma~\ref{orbit correspondence D7 B6} the $H$-orbits on $\Lambda$ correspond to the $S$-orbits on $A$. Since $S'\leq S$ we have $\#orb(H,\Lambda) \leq \#orb(S',A)$, where $\#orb(M,Y)$ denotes the number of orbits of a group $M$ on a set $Y$. This proves our lemma.
\end{proof}

For each of the $9$ cases in Proposition~\ref{list of representatives} we shall find a subgroup $S'$ of the stabilizer $S$ of the spinor, that thanks to Lemma~\ref{subgroup stab is enough} will be enough to prove Theorem~\ref{theorem B6 singular}.

The idea is to use the stabilizers listed in Theorem~\ref{stabs in char 0} when $p=0$, argue that such subgroups naturally exist if $p=2$, and that they are contained in the full stabilizers. From Lemmas $23$ to $29$ in \cite{popov} we see the explicit lists of generators for the semisimple and unipotent parts of each of the stabilizers of the representatives in Theorem~\ref{stabs in char 0}. These are all given in terms of root elements $s_{i,j}(\lambda)$ and they naturally make sense if $p=2$. As an example consider the stabilizer $U_{26}.SL_4$. The semisimple part $SL_4$ is generated by the elements $s_{2,10}(\lambda)s_{4,13}(\lambda), s_{6,8}(\lambda)s_{3,12}(\lambda)$ and $s_{5,7}(\lambda)s_{11,13}(\lambda)$ for $\lambda\in k^*$. The unipotent radical is similarly well defined in characteristic $2$. 

Because of some explicit work done in \cite{Gatti} we are actually going to switch to the list of $D_7$-orbit representatives on $X$ given in \cite[Table $1$]{Gatti}. In order to make it easy for the reader we adopt the notation in \cite{Gatti} that reverses the roles of the elements $e_i$ and $f_i$. Therefore from now on all spinor representatives are given in terms of $C_{\langle f_1,\dots f_7\rangle}$, so that all the explicit unipotent generators and subspaces decompositions match \cite{Gatti}. In Table~\ref{tab:Stabilizers structure in characteristic $0$} we have a description of the connected components of stabilizers of the spinor representatives for the $D_7$-orbits on $X$, as in \cite[\S $2.2$]{Gatti}. In particular we describe the action of the semisimple part on $V=V_{14}$. We do not list the unipotent generators for the unipotent radical, which we will later be referring to.

\begin{table}[h]
\centering
\  \begin{tabular}{||c c c  ||} 
 \hline
 Orbit type &  $(D_7)_x^0$ & Semisimple structure  \\ [0.5ex] 
 \hline\hline
  $2$ & $U_{21} SL_7$ & $\langle e_1,e_2,e_3,e_4,e_5,e_6,e_7 \rangle \oplus \langle f_1,f_2,f_3,f_4,f_5,f_6,f_7 \rangle$ \\ 
   & & $V_{A_6}(\lambda_1)\oplus  V^*_{A_6}(\lambda_1)$  \\
 \hline
 $3$ & $U_{27}(SL_3\times Spin_7)$&  $\langle e_5,e_6,e_7 \rangle \oplus \langle f_5,f_6,f_7 \rangle\oplus \langle e_1,e_2,e_3,e_4,f_4,f_3,f_2,f_1\rangle$    \\
 & & $V_{A_2}(\lambda_1)\oplus  V^*_{A_2}(\lambda_1)\oplus V_{B_3}(\lambda_3)$  \\
 \hline
 $4$& $U_{12} SL_6$& $\langle e_1,e_2,e_3,e_4,e_5,e_6 \rangle \oplus \langle f_1,f_2,f_3,f_4,f_5,f_6 \rangle\oplus \langle e_7\rangle\oplus \langle f_7\rangle$\\ 
  & & $V_{A_5}(\lambda_1)\oplus  V^*_{A_5}(\lambda_1)\oplus \langle e_7\rangle\oplus \langle f_7\rangle$  \\
 \hline
$5$ &$  U_{21} (SL_3\times SL_3) $&  $\langle e_1,e_2,e_3 \rangle \oplus \langle f_1,f_2,f_3 \rangle\oplus \langle e_4,e_5,e_6 \rangle \oplus \langle f_4,f_5,f_6 \rangle\oplus \langle e_7 \rangle \oplus \langle f_7 \rangle$\\
& & $V_{A_2}(\lambda_1)\oplus  V^*_{A_2}(\lambda_1)\oplus V_{A_2}(\lambda_1)\oplus  V^*_{A_2}(\lambda_1)$  \\
 \hline
 $6$ & $G_2\times G_2$& $\langle e_1,e_2,e_3,f_1,f_2,f_3,e_7+f_7\rangle \oplus \langle e_4,e_5,e_6,f_4,f_5,f_5,e_7-f_7 \rangle $  \\ 
  & & $V_{G_2}(\lambda_1)\oplus  V^*_{G_2}(\lambda_1)\oplus V_{B_3}(\lambda_3)$  \\
 \hline
 $7$ & $U_{19} (SL_2\times_{\mathbf{Z}_2}Sp_4) $&  $\langle e_2,e_3,e_5,e_6,f_2,f_3,f_5,f_6 \rangle \oplus \langle f_4,f_7,e_1 \rangle\oplus \langle e_4,e_7,f_1 \rangle$ \\ 
   & & $V_{A_1}(\lambda_1)\otimes V_{C_2}(\lambda_1)\oplus  V_{A_1}(2\lambda_1)\oplus V^*_{A_1}(2\lambda_1)$  \\
 \hline
 $8$ &$ U_{14} G_2$& $\langle f_1,f_2,e_3,f_4,e_5,e_6,e_7\rangle \oplus \langle e_1,e_2,f_3,e_4,f_5,f_6,f_7\rangle$  \\ 
    & & $V_{G_2}(\lambda_1)\oplus  V^*_{G_2}(\lambda_1)$  \\
 \hline
 $9$ & $U_{26} (Sp_6\times_{\mathbf{Z}_2}T_1)$&  $ \langle e_1,e_6,f_4,f_3,e_5,e_2 \rangle \oplus \langle f_1,f_6,e_4,e_3,f_5,f_2 \rangle\oplus \langle e_7\rangle\oplus \langle f_7\rangle$\\ 
  & & $V_{C_3}(\lambda_1)\oplus  V^*_{C_3}(\lambda_1)\oplus \langle e_7\rangle\oplus \langle f_7\rangle$  \\
 \hline
  $10$& $U_{26}SL_4$& $\langle e_2,e_7,e_5,f_3 \rangle \oplus \langle f_2,f_5,e_3,e_4 \rangle\oplus \langle f_1,f_4,f_6,e_6,e_4,e_1\rangle$  \\ 
   & & $V_{A_3}(\lambda_1)\oplus  V^*_{A_3}(\lambda_1)\oplus V_{A_3}(\lambda_2)$  \\
 \hline
\end{tabular}
\caption{Stabilizers structure in characteristic $0$}
\label{tab:Stabilizers structure in characteristic $0$}
\end{table}

Finally here are the orbit representatives in \cite[Table $1$]{Gatti} together with subgroups that fix them defined for $p=2$.

\begin{lemma}\label{almost stabs in characteristic 2}
If $p=2$, the subgroups in Table~\ref{tab:Subgroups of stabs if $p=2$}, defined analogously to the ones in Table~\ref{tab:Stabilizers structure in characteristic $0$},  are contained in the stabilizers of the orbit representatives listed in Table~\ref{tab:Subgroups of stabs if $p=2$}:
\begin{table}[h]
\begin{center}
 \begin{tabular}{||c c c  ||} 
 \hline
 Orbit type & Spinor $x$ & $S'\leq (D_7)_x$  \\ [0.5ex] 
 \hline\hline
  $2$ & $ 1 $ & $U_{21} SL_7$  \\ 
 \hline
 $3$ &$ 1+f_1f_2f_3f_4$& $U_{27}(SL_3\times Spin_7)$   \\
 \hline
 $4$&$ 1+f_1f_2f_3f_4f_5f_6$& $U_{12} SL_6$   \\
 \hline
$5$ &$ 1+f_1f_2f_3f_7+f_1f_2f_3f_4f_5f_6$& $U_{21} (SL_3\times SL_3) $  \\
 \hline
 $6$ &$\lambda (1+f_1f_2f_3f_7+f_4f_5f_6f_7+f_1f_2f_3f_4f_5f_6):\lambda\in k^*$& $G_2\times G_2$   \\ 
 \hline
 $7$ &$1+f_1f_2f_3f_7+f_1f_5f_6f_7+f_1f_2f_3f_4f_5f_6$& $U_{19} (SL_2\times Sp_4) $   \\ 
 \hline
 $8$ &$ 1+f_1f_2f_3f_7+f_1f_5f_6f_7+f_2f_4f_6f_7+f_1f_2f_3f_4f_5f_6$&$U_{14} G_2$  \\ 
 \hline
 $9$ &$ 1+f_1f_2f_3f_4+f_3f_4f_5f_6$& $U_{26} (Sp_6\times T_1)$   \\ 
 \hline
  $10$&$1+f_1f_2f_3f_4+f_3f_4f_5f_6+f_1f_3f_6f_7$& $U_{26}SL_4$   \\ 
 \hline
\end{tabular}
\end{center}
\caption{Subgroups of stabilizers if $p=2$}
\label{tab:Subgroups of stabs if $p=2$}
\end{table}
\end{lemma}

Now that we have a list of subgroups contained in the stabilizers of the spinors of types $2$ to $10$, we can adopt the strategy outlined in Lemma~\ref{subgroup stab is enough} and preceding discussion. As it is somewhat more complicated then the other cases, we start by determining how the dense orbit of $D_7$ on $1$-spaces, i.e. the one with connected stabilizer $G_2G_2$, splits when restricting to $B_6$. 

We recall some facts about $G_2$. We can identify $G_2$ as a subgroup of $SO_7$ via the Weyl module $W_{G_2}(\lambda_1)$. When $p=2$, we can also identify $G_2$ as a subgroup of $Sp_6$, thanks to the embedding of $SO_7$ in $Sp_6$. Given a natural module $V_8=\langle e_1,e_2,e_3,e_4,f_1,f_2,f_3,f_4\rangle$ for $SO_8$, we say that $(e_4+f_4)^\perp$ is the natural module $V_7$ for $SO_7=(SO_8)_{e_4+f_4}$. By $N_1$ we denote the stabilizer of a non-singular $1$-space.

\begin{lemma}\label{facts about g2}
Let $G_2\leq SO_7$ with $p=2$ and $V_6,V_7$ be the natural modules for $Sp_6$ and $SO_7$ respectively. Let $SO_6\leq Sp_6$ be a subgroup of $Sp_6$ stabilizing a non degenerate quadratic form on $V_6$. Then the following statements are true:

\begin{enumerate}[label=(\roman*)]
\item $G_2\cap SO_6=SL_3.2$;
\item $G_2$ is transitive on $1$-spaces in $V_6$, i.e. $Sp_6=G_2P_1$ for a $P_1$ parabolic subgroup of $Sp_6$;
\item $G_2$ is transitive on both singular and non-singular $1$ spaces in $V_7$, i.e. $SO_7=G_2P_1$ and $SO_7=G_2N_1$ .
\end{enumerate}
\end{lemma}

\begin{proof}
Statements $(ii),(iii)$ follow by \cite[Theorem B]{factorizations}. Now \cite[Theorem B]{factorizations} tells us that $Sp_6=G_2SO_6$, which implies $\dim (G_2\cap SO_6)=8$. 
Since $SL_3\leq G_2\cap SO_6$, we have $(G_2\cap SO_6)^0=SL_3$. If we consider an element of order $2$ in $G_2$ swapping the two totally singular $3$-spaces $\langle e_1,e_2,e_3 \rangle$ and $\langle f_1,f_2,f_3 \rangle$ then the subgroup $SL_3.2$ is contained in $G_2\cap SO_6$. Any other element of $G_2\cap SO_6\setminus SL_3$ normalizing $SL_3$ must in fact swap the two $3$-spaces, proving that $G_2\cap SO_6=SL_3.2$.
\end{proof}

Recall that we identified $H=B_6$ with the stabilizer in $D=D_7$ of the non-singular $1$-space $\langle e_7+f_7\rangle=\langle w_1 \rangle$. Note that since $p=2$ stabilizers of non-singular $1$-spaces are the same as stabilizers of non-singular vectors. Let $V_{13}$ denote $(w_1)^{\perp}$, the $13$-dimensional space stabilized by $H$.

Let $D_4,D_4^*$ be the pointwise stabilizers in $D_7$ of $V_6^*:=\langle e_4,e_5,e_6,f_4,f_5,f_6\rangle$ and $V_6:=\langle e_1,e_2,e_3,f_1,f_2,f_3\rangle$ respectively. Let $V_7=V_6\oplus \langle w_1 \rangle $ and let $B_3=(D_4)_{w_1}$ and $B_3^*=(D_4^*)_{w_1}$.
These two subgroups of $B_6$ isomorphic to $B_3$ intersect trivially, and we have $B_3B_3^*\leq B_6\leq D_7$. Since $G_2$ is naturally a subgroup of $B_3$ when $p=2$, we have $G_2G_2\leq B_3B_3^*$.

\begin{lemma}\label{G2A2}
Let $w_2$ be a non-singular vector such that $(w_1,w_2)\neq 0$. Then $(G_2G_2)_{w_2}$ is isomorphic to $(A_2.2)(A_2.2)$.
\end{lemma}

\begin{proof}
We first show that $(B_3B_3^*)_{w_2}=D_3D_3$, and then take the intersection with $G_2G_2$. It is sufficient to show that $(B_3)_{w_2}=D_3$ for each factor $B_3$. In order to do this we show that we can define a non-degenerate quadratic form $Q'$ on $\overline{V_6}:=V_7/\langle w_1\rangle$ that is fixed by $(B_3)_{w_2}$ and such that any $g\in B_3$ fixing it must fix $w_2$. We can clearly assume that $(w_1,w_2)=1$.

Let $v=v_6+\langle w_1 \rangle$ for some $v_6\in V_6$ and define $Q'(v)=Q(v_6)+(v_6,w_2)^2$. Expand the definition of $Q'$ to $V_7$ by setting $Q'(v_7):=Q'(v_7+\langle w_1 \rangle)$ for any $v_7\in V_7$.

Then if $g\in (B_3)_{w_2}$ we have $v_6 g=v_6' +\lambda w_1$ for $v_6 '\in \langle e_1, f_1,e_2,f_2,e_3,f_3\rangle,\lambda\in k$. Now $Q'(v g)=Q(v_6 ')+(v_6',w_2)^2=Q(v_6 g+\lambda w_1)+(v_6 g+\lambda w_1,w_2)^2=Q(v_6)+\lambda^2+(v_6,w_2)^2+\lambda^2 (w_1,w_2)=Q'(v)$, as wanted. Next observe that if $(\cdot,\cdot)'$ is the bilinear form corresponding to $Q'$, for any $u=u_6+\langle w_1 \rangle$ we have $(u,v)'=Q'(u)+Q'(v)+Q'(u+v)=Q(u_6)+Q(v_6)+Q(u_6+v_6)+(u_6,w_2)^2+(v_6,w_2)^2+(u_6+v_6,w_2)^2=Q(u_6)+Q(v_6)+Q(u_6+v_6)=(u_6,v_6)$, so that the quadratic form is non-degenerate.

Now suppose that $g\in B_3$ fixes the quadratic form $Q'$ on $\overline{V_6}$. For $v_6\in V_6$ $(w_2 g, v_6)^2=(w_2,v_6g^{-1})^2=Q(v_6g^{-1})+Q'(v_6g^{-1})=Q(v_6)+Q'(v_6)=(w_2,v_6)^2$ which implies that 
$(w_2 g, v_6)=(w_2,v_6)$. 

To conclude write $w_2=u_6+u_6^*+\alpha e_7+\beta f_7$ for $u_6\in V_6$ and $u_6^*\in V_6^*$ and note that since $g\in B_3$, $u_6^*g=u_6^*$. This means that not only $(w_2g,v_6)=(w_2,v_6)$ for all $v_6\in V_6$, but also $(w_2g,v_6^*)=(w_2,v_6^*)$ for all $v_6^*\in V_6^*$. Since we also have $(w_2,w_1)=(w_2g,w_1)$ and $0=(w_2,w_2)=(w_2,w_2g)$, this implies that $w_2$ is fixed by $g$.
Therefore  $(B_3)_{w_2}=D_3$ as wanted.

By Lemma~\ref{facts about g2}, $G_2\cap D_3 = A_2.2$ and we are done.
\end{proof}

\begin{lemma}\label{G2P1}
Let $w_2$ be a non-singular vector in $V=V_{14}$ such that $(w_1,w_2)=0$. Let $P_1'$ be the stabilizer in $G_2$ of a non-zero vector in its natural $6$-dimensional representation. Then $(G_2G_2)_{w_2}$ is conjugate to either $P_1'P_1'$, $P_1'G_2$ or $G_2G_2$.
\end{lemma}

\begin{proof}
If $w_2=w_1$ we have $(G_2G_2)_{w_2}=G_2G_2$. Now let $w_2=v+aw_1$ for $v\neq 0$ in $V_{12}:=\langle e_1,\dots e_6,f_1\dots f_6\rangle$ and note that any element of $H$ fixing such a vector must fix $v$, since $w_1$ is fixed. At the same time any element of $H$ fixing $\overline{w_2}=v+\langle w_1 \rangle$ must actually stabilize $v$. This is because if $vh= v+aw_1$, we must have $Q(v)=Q(v+aw_1)=Q(v)+a^2$ which implies $a=0$. We therefore have that $(B_3B_3)_{w_2}= (C_3C_3)_{\overline{w_2}}=(C_3C_3)_{\overline{v}}$. 

Now write $v=v_1+v_2$ for $v_1\in V_6$ and $v_2\in  V_6^*$. We then have $(C_3C_3)_{\overline{v}}=(C_3)_{\overline{v_1}}(C_3)_{\overline{v_2}}$. The two cases $(v_1,v_2\neq 0)$ and ($v_1=0$ or $v_2=0$) together with taking the intersection with $G_2G_2$ give the result.
\end{proof}

We are now finally able to write down the orbit representatives of $G_2G_2$ acting on non-singular $1$-spaces.

\begin{lemma}\label{G2G2}
The connected stabilizer $G_2G_2$ has the following orbits on non-singular $1$-spaces in $V=V_{14}$:

\begin{center}
 \begin{tabular}{||c c c ||} 
 \hline
 Orbit type & Rep $v$& $H_v$ \\ [0.5ex] 
 \hline\hline
  $a$ & $ \alpha e_7+\alpha^{-1}f_7 :\alpha\neq 0,1$ & $(A_2.2)(A_2.2)$ \\ 
 \hline
   $b$ & $ e_7+f_7 $ & $G_2\times G_2$ \\ 
    \hline
   $c$ & $f_4+e_4$ & $G_2P_1'$ \\ 
    \hline
   $d$ & $ e_1+f_1 $ & $P_1'G_2$ \\ 
      \hline
   $e$ & $ e_1+f_1+e_4 $ & $P_1'P_1'$ \\ 
 \hline
\end{tabular}
\end{center}
\end{lemma}

\begin{proof}
This follows from Lemmas \ref{G2A2} and \ref{G2P1}.
\end{proof}

We now do the same type of work on the other $D_7$-orbits in Table~\ref{tab:Subgroups of stabs if $p=2$}.

\begin{proposition}\label{orbits on non-singular 1-spaces}
The groups $S'\leq D_7$ listed in Table~\ref{tab:Subgroups of stabs if $p=2$} have the following orbits on vectors in $\{v\in V_{14}:Q(v)=1\}$:

\begin{center}
 \begin{tabular}{||c c c ||} 
 \hline
 $S'$ & Rep $v$& $S'_v$ \\ [0.5ex] 
 \hline\hline
  $U_{21}. SL_7$  & $ e_1+f_1 $ & $ U_{21}. SL_6$ \\
 \hline
 $U_{26}. (Sp_6 T_1)$   &$ e_1+f_1 $ & $ U_{25}.(Sp_4 T_1)$ \\ 
    & $ e_7+f_7 $ & $U_{14}.Sp_6 $ \\ 
 \hline
   $U_{27}.(Spin_7 SL_3)$  &   $ e_5+f_5 $ & $U_{23}.(Spin_7 SL_2)$ \\ 
    & $ e_1+f_1 $ & $U_{21}.(G_2 SL_3)$ \\ 
 \hline
 $U_{12}.SL_6$ & $ e_1+f_1 $ & $U_{11}.SL_5$ \\ 
    &$ \alpha e_7+\alpha^{-1}f_7 :\alpha\neq 0$ & $SL_6$ \\ 
 \hline
$U_{21} .(SL_3 SL_3) $  & $ e_1+f_1 $ & $U_{16}.(SL_2SL_3)$ \\ 

   & $ e_4+f_4$ & $U_{16}.(SL_2SL_3)$ \\ 

    & $ \alpha e_7 +\alpha^{-1}f_{7}:\alpha\neq 0 $ & $U_{15}.(SL_3SL_3)$ \\ 

   & $ e_1 +f_1+f_4$ & $U_{18}.(SL_2SL_2)$ \\ 
 \hline
 $G_2G_2$    & $ \alpha e_7+\alpha^{-1}f_7 :\alpha\neq 0,1$ & $(SL_3.2)(SL_3.2)$ \\ 
 
    & $ e_7+f_7 $ & $G_2\times G_2$ \\ 
    
  & $f_4+e_4$ & $G_2P_1'$ \\ 
  
    & $ e_1+f_1 $ & $P_1'G_2$ \\ 
     
   & $ e_1+f_1+e_4 $ & $P_1'P_1'$ \\  
 \hline
$U_{19} .(SL_2 Sp_4) $    & $ e_1+f_1 $ & $U_{10}.Sp_4$ \\ 
  & $ e_5+f_5$ & $U_{16}.(SL_2SL_2)$ \\ 
 & $\alpha e_4+ \alpha^{-1} f_4 :\alpha\neq 0 $ & $U_9.(Sp_4T_1.2)$ \\ 
 \hline
 $U_{14} .G_2$    & $ e_1+f_1 $ & $U_{14}.SL_2$ \\ 
 & $ \alpha e_7 +\alpha^{-1}f_{7}:\alpha\neq 0 $ & $U_{8}.(SL_3.2)$ \\ 
 \hline

$U_{26}.SL_4$     & $ e_1+f_1 $ & $U_{22}.Sp_4 $ \\ 

   & $ e_7+f_7 $ & $U_{20}.SL_3$ \\ 
 \hline
\end{tabular}
\end{center}
 \end{proposition}

\begin{proof}
Note that the $G_2G_2$ case has been dealt with in Lemma~\ref{G2G2}. When $p=0$, the $S'$-orbits on non-singular $1$-spaces have in fact already calculated in \cite[2.2]{Gatti}. The only difference if $p=2$ occurs precisely in the $G_2G_2$ case. Given the lack of details for the proofs in \cite{Gatti}, we give an outline of the proof. For reference note that the dimensions of a couple of unipotent radicals were miscalculated in \cite{Gatti} and have now been fixed. Furthermore it is possible to check that when passing to finite fields all sizes of the orbits add up to the number of non-singular $1$-spaces.
 
The case $S'=U_{21}.SL_7$ is trivial. Recall that all the subgroups $S'$ listed have structure as in Table~\ref{tab:Stabilizers structure in characteristic $0$} and the corresponding unipotent radicals are as in \cite[2.2]{Gatti}.

We will provide details for the next case, with the rest following similarly.
We proceed to find the orbits of $U_{26}.(Sp_6T_1)$ (orbit $(9)$ in Table~\ref{tab:Subgroups of stabs if $p=2$}) on non-singular $1$-spaces. The semisimple part $Sp_6\times T_1$ stabilizes $\langle e_1,e_6,f_4,f_3,e_5,e_2 \rangle \oplus \langle f_1,f_6,e_4,e_3,f_5,f_2 \rangle\oplus \langle e_7\rangle\oplus \langle f_7\rangle$, with $Sp_6$ acting  on $\langle e_1,e_6,f_4,f_3,e_5,e_2 \rangle$ and its dual $\langle f_1,f_6,e_4,e_3,f_5,f_2 \rangle$, and $T_1$ acting by scalar multiplication on $e_7$ and its dual $f_7$. 

The stabilizer of $e_1+f_1$ in $Sp_6$ is a Levi subgroup $Sp_4$, since an element of $Sp_6$ fixes $e_1+f_1$ if and only if it fixes $e_1$ and stabilizes $\langle e_6,f_4,f_3,e_5,e_2 \rangle$. 
A set of $26$ root subgroups generating the unipotent radical $U_{26}$ (written as elements of $C$), is the following (see \cite[2.2, III.]{Gatti}):
$1+\lambda y$ for $y$ equal to $e_7e_i,e_7f_i$ for $i\neq 7$; $e_1f_3,e_1f_4,e_1e_5,e_1e_6,e_2f_3,e_2f_4,e_2e_5,e_2e_6,f_3e_5,f_3e_6,f_4e_5,f_4e_6$; and $(1+\lambda e_5e_6)(1+\lambda f_4f_4),(1+\lambda e_1e_2)(1+\lambda f_3f_4)$, where $\lambda \in k^*$.

The unipotents in the list of generators of $U_{26} $ that do not fix $e_1+f_1$ are $1+e_7e_1,1+e_7f_1,1+e_1f_3,1+e_1f_4,1+e_1e_5,1+e_1e_6,(1+e_1e_2)(1+f_3f_4)$. Since $(1+e_7e_1)(1+e_7f_1)$ fixes $e_1+f_1$ and there is no other product of the excluded generators that fixes $e_1+f_1$, we find that $(U_{26})_{e_1+f_1}=U_{20}$. Furthermore there exists a unipotent element $g$ of $Sp_6$ such that $e_1g=e_1+f_3$ and $f_1g=f_1$, giving that $g(1+e_1f_3)$ fixes $e_1+f_1$. In the same fashion elements $g_2,g_3,g_4,g_5$ in $Sp_6$ can be found so that $g_2(1+e_1f_4),g_3(1+e_1e_5),g_4(1+e_1e_6)$ and $g_5(1+e_1e_2)(1+f_3f_4)$ all fix $e_1+f_1$. Finally note that since the semisimple part $Sp_6T_1$ fixes $\langle e_7\rangle$, there is no element $g\in Sp_6T_1$ such that $g(1+e_7e_1)$ or $g(1+e_7f_1)$ fixes $e_1+f_1$.
This shows that the unipotent radical of $(U_{26}.(Sp_6 T_1))_{e_1+f_1}$ has dimension $20+5=25$ and therefore $(U_{26}.(Sp_6 T_1))_{e_1+f_1}=U_{25}.(Sp_4 T_1)$.

For $e_7+f_7$ we immediately see that $(Sp_6\times T_1)_{e_7+f_7}=Sp_6$.  From the list of generators of the unipotent radical $U_{26}$ we find that the $12$ elements $1+e_7e_i,1+e_7f_i$ do not fix $e_7+f_7$. It is then easy to see that $(U_{26})_{e_7+f_7}=U_{14}$ and since $Sp_6\times T_1$ stabilizes $\langle e_7,f_7\rangle$ we in fact have $(U_{26}.(Sp_6 T_1))_{e_7+f_7}=U_{14}.Sp_6$.

To check that the vectors $e_1+f_1$ and $e_7+f_7$ form a complete set of orbit representatives for $U_{26}.(Sp_6 T_1)$ acting on non-singular $1$-spaces in $V_{14}$, we pass to finite fields. The stabilizers that we found are connected and therefore we have orbits with stabilizers $q^{25}.Sp_4(q)\times (q-1)$ and $q^{14}.Sp_6(q)$, of $q^{26}.Sp_6(q)\times (q-1)$ acting on $V_{14}(q)$. A simple calculation with the sizes of the stabilizers shows that these two orbits do indeed contain all the elements $v\in V$ such that $Q(v)=1$, and we conclude. 

The remaining cases follow extremely similarly, with the reduction to finite fields completing the analysis. 
Here are a couple of highlights of the reduction to finite fields.

For $S'=U_{14}.G_2$, by Lang-Steinberg we can see that when passing to finite fields, the orbit with stabilizer $U_{8}.(A_2.2)$ splits into two orbits with stabilizers $q^8.SU_3(q).2$ and $q^8.SL_3(q).2$. Similarly, for $S'=U_{19} .(SL_2 Sp_4)$, when going to finite fields the orbit with stabilizer $U_9.(Sp_4T_1.2)$ splits into two orbits with stabilizers $q^9.(Sp_4(q)\times (q-1).2)$ and $q^9.(Sp_4(q)\times (q+1).2$. 
\end{proof}

\subsection{Completion of the proof of Theorem~\ref{theorem B6 singular}}

We proceed with last step of the proof of Theorem~\ref{theorem B6 singular}. By Proposition~\ref{orbits on non-singular 1-spaces} there are only $5$ $D_7$-orbits on $1$-spaces in $X$ that split into infinitely many $B_6$-orbits. We will show that these $5$ $D_7$-orbits do not contain any singular $1$-spaces. We list the families of $S'$-orbits on non-singular $1$-spaces, for the subgroups $S'$ corresponding to these $5$-orbits in Table~\ref{tab:orbits that split into infinitely many $B_6$ orbits}. 
By Lemma~\ref{orbit correspondence D7 B6}, in order to find an explicit expression for a spinor representative when restricting the action of $D_7$ to $B_6$,  given a representative $v\in V_{14}$ of an orbit on non-singular $1$-spaces, we need to find an element $g\in D_7$ such that $(e_7+f_7)g=v$.
Set $\lambda=\sqrt{\alpha}$.

\begin{table}[h]
\begin{center}
 \begin{tabular}{||c c c ||} 
 \hline
 Spinor representative & Rep $v\in V_{14}$& $g\in G:(e_7+f_7)g=v$ \\ [0.5ex] 
 \hline\hline
 $1+f_1f_2f_3f_7+f_4f_5f_6f_7+f_1f_2f_3f_4f_5f_6$ &  $\alpha e_7 + \alpha^{-1}f_7 : \alpha\neq 1$ & $s_7(\lambda)$\\
 \hline
  $1+f_1f_2f_3f_4f_5f_6$ & $ \alpha e_7 + \alpha^{-1}f_7 $ & $s_7(\lambda)$  \\ 
 \hline
  $ 1+f_1f_2f_3f_7+f_1f_5f_6f_7+f_2f_4f_6f_7+f_1f_2f_3f_4f_5f_6$ &  $ \alpha e_7 + \alpha^{-1}f_7 $ & $s_7(\lambda)$  \\ 
 \hline
   $1+f_1f_2f_3f_7+f_1f_2f_3f_4f_5f_6$ &  $ \alpha e_7 + \alpha^{-1}f_7 $ & $s_7(\lambda)$  \\ 
   \hline
   $1+f_1f_2f_3f_7+f_1f_5f_6f_7+f_1f_2f_3f_4f_5f_6$ &  $ \alpha e_4 + \alpha^{-1}f_4 $ & $s_{7,11}(1)s_{4,14}(1)s_{4,7}s_4(\lambda)$  \\ 
 \hline
\end{tabular}
\end{center}
\caption{$D_7$-orbits that split into infinitely many $B_6$ orbits}
\label{tab:orbits that split into infinitely many $B_6$ orbits}
\end{table}

We compute the action of the elements $g$ listed on the corresponding spinors, remembering that $s_i(\lambda)=\lambda^{-1}+(\lambda+\lambda^{-1})e_if_i$. We find that: 
\begin{enumerate}[label=(\roman*)]
\item $s_7(\lambda)\cdot (1+f_1f_2f_3f_7+f_4f_5f_6f_7+f_1f_2f_3f_4f_5f_6)=\lambda^{-1}(1+f_1f_2f_3f_4f_5f_6)+\lambda(f_1f_2f_3f_7+f_4f_5f_6f_7)$
\item $s_7(\lambda)\cdot (1+f_1f_2f_3f_4f_5f_6)=\lambda^{-1}(1+f_1f_2f_3f_4f_5f_6)$;
\item $s_7(\lambda)\cdot (1+f_1f_2f_3f_7+f_1f_5f_6f_7+f_2f_4f_6f_7+f_1f_2f_3f_4f_5f_6)= \lambda^{-1}(1+f_1f_2f_3f_4f_5f_6)+\lambda(f_1f_2f_3f_7+f_1f_5f_6f_7+f_2f_4f_6f_7)$
\item $s_7(\lambda)\cdot (1+f_1f_2f_3f_7+f_1f_2f_3f_4f_5f_6)=\lambda^{-1}(1+f_1f_2f_3f_4f_5f_6)+\lambda f_1f_2f_3f_7+$
\item $s_4(\lambda)s_{4,7}(1)s_{4,14}(1)s_{7,11}(1)\cdot (1+f_1f_2f_3f_7+f_1f_5f_6f_7+f_1f_2f_3f_4f_5f_6)=\lambda^{-1}(1+f_1f_2f_3f_7+f_1f_5f_6f_7+f_1f_2f_3f_5f_6f_7)+\lambda(f_1f_2f_3f_4f_5f_6)$
\end{enumerate}

We are finally ready to conclude the proof of Theorem~\ref{theorem B6 singular}.
Given the description of the quadratic form in the introduction of this section we can see that the spinors on the right hand sides of the above equations are non-singular. For example in the first case note that $Q(1+f_1f_2f_3f_4f_5f_6)=Q(f_1f_2f_3f_7+f_4f_5f_6f_7))=1$, giving  $Q(\lambda^{-1}(1+f_1f_2f_3f_4f_5f_6)+\lambda(f_1f_2f_3f_7+f_4f_5f_6f_7))=(\lambda^{-1}+\lambda)^2$ which is non zero when $\lambda\neq 0,1$, as by assumption. By Lemma~\ref{subgroup stab is enough} this means that the $5$ $D_7$-orbits with representatives as in Table~\ref{tab:orbits that split into infinitely many $B_6$ orbits} only contain non-singular vectors.
This concludes our analysis and proves Theorem~\ref{theorem B6 singular}. Note that the fact that a generic stabilizer is $P_1'P_1'<G_2G_2$ follows by dimension considerations and Lemma~\ref{G2G2}.

\section{Conclusion of the proof of Theorem~\ref{main theorem} when $H$ is simple}\label{conclusion of H simple}

In this section we complete the proof of Theorem~\ref{main theorem} in the case where the subgroup $H\leq SO(V)$ is simple. By Theorem~\ref{simple candidates} all that remains to be proven is the following:

\begin{proposition}\label{simple conclusion big 0-space}
Let $H\leq SO(V)$ be a simple connected algebraic group. 
Suppose that one of the following holds:
\begin{enumerate}[label=(\roman*)]
\item $V$ is in Table~\ref{tab:Adjoint modules};
\item $H=C_n(n\geq 3),V=V_H(\lambda_2)$;
\item $H=F_4$, $V=V_H(\lambda_4)$ or $V(\lambda_1)(p=2)$.
\end{enumerate}

If $V$ is not a finite orbit module and $H$ has a dense orbit on singular $1$-spaces in $V$, then $(H,V)$ is one of the following:

 \begin{center}
 \begin{tabular}{||c c c c||} 
 \hline
 $H$ &$V$ & $\dim V$ & $p$ \\ [0.5ex] 
 \hline\hline
  $A_2$ & $\lambda_1+\lambda_2$ &  $8$& $\neq 3$ \\ 
   $A_3$ &$\lambda_1+\lambda_3$   &$14$ & $2$ \\ 
 \hline
 $B_2$&$\lambda_2$  & $10$ & $\neq 2$ \\ 
 \hline
 $C_3$& $\lambda_2$ &  $14$& $\neq 3$ \\ 
    $C_4$& $\lambda_2$ & $26$ & $2$ \\ 
\hline
 $D_4$& $\lambda_2$ & $26$ & $2$ \\ 
\hline
 $G_2$& $\lambda_2$ &  $14$& $\neq 3$ \\ 
  \hline
 $F_4$&$\lambda_4 $  &$26$  & $\neq 3$ \\
  $F_4$&  $\lambda_1 $&  $26$& $p=2$ \\
  \hline
 \end{tabular}
\end{center} 

Conversely for any such $(H,V)$, $V$ is a finite singular orbit module.
\end{proposition}

Let $H$ be as in Proposition~\ref{simple conclusion big 0-space}. Let $T$ be a maximal torus of $H$, and $W=N_H(T)/T$ the Weyl group.  We will be able to reduce drastically the number of cases to analyze, thanks to the following lemma. 

\begin{lemma}\cite[Lemma 2.1]{finite}\label{zeroSpace}
Let $v,v'\in V_0$, the zero weight space of $V$ relative to $T$. Then $v$ and $v'$ are in the same $H$-orbit if and only if they are in the same $W$-orbit. 
\end{lemma}

\begin{corollary}\label{zero3}
If $\dim V_0\geq 3$ then $H$ has infinitely many orbits on singular $1$-spaces in $V$.
\end{corollary}
\begin{proof}
We start by noting that since $V_0$ is perpendicular to weight spaces for non-zero weights, $V_0$ must be a non-degenerate space. Since $\dim V_0\geq 3$ it must contain infinitely many singular $1$-spaces. But since the Weyl group $W$ is finite, Lemma~\ref{zeroSpace} implies that these singular $1$-spaces belong to infinitely many different orbits. 
 \end{proof}
 
 We are also interested in dealing with the possibility of having a dense orbit. It is sufficient to only slightly adapt \cite[Lemma 2.3]{finite} for singular $1$-spaces.
 
 \begin{lemma}\label{dimensionConditions}
 Let $V_0$ be the zero weight space of $V$ relative to $T$. Let $C=C_H(V_0)^0$. Suppose that $\dim H-\dim C=\dim V-\dim V_0$ and $\dim V_0\geq 3$. Then $H$ has no dense orbit on singular $1$-spaces in $V$.
 \end{lemma}
 
 \begin{proof}
 Assume that $H$ has a dense orbit on singular $1$-spaces in $V$. Note that since $\dim V_0\geq 3$ there are infinitely many singular $1$-spaces in $V_0$. If there is a representative of the dense orbit in $V_0$ then the intersection of the dense orbit and $P_1(V_0)$ is dense in the set of singular $1$-spaces in $V_0$, but this is absurd since by Lemma~\ref{zeroSpace} there cannot be infinitely many elements of $P_1(V_0)$ in the same orbit.
 
 Therefore, all we need to show is that under the assumptions of the lemma there is a representative of the dense orbit in $P_1(V_0)$. If $v\in V_0$ then $C_H(v)$ contains $T$, so there are only finitely many possibilities for $C_H(v)$. Hence $\{v\in V_0:C< C_H(v)^0\}=\bigcup_{C<D=D^0}C_{V_0}(D)$ is a finite union of proper closed subsets of $V_0$ and is therefore closed. It follows that $\Delta=\{v\in V_0:C_H(v)^0=C\}$ is an open dense subset of $V_0$. Since $\dim V_0\geq 3$ we know that $\Delta$ intersects the singular vectors in $V_0$ in an open dense subset.
 
 We now show that there is some singular $v_0\in\Delta$ such that $\langle v_0\rangle$ is in the dense orbit on singular $1$-spaces. Let $\phi:H\times \Delta\rightarrow V $ be the morphism $(h,v)\rightarrow h(v)$. Pick any singular $v_0\in \Delta$. Then by Lemma~\ref{zeroSpace}, $\phi^{-1}(v_0)$ has component $\{(h,v_0):h\in C_H(v_0)^0=C\}$, and hence $\dim \phi^{-1}(v_0)=\dim C$. Therefore $$\dim Im(\phi)=\dim H+\dim V_0-\dim C=\dim V$$ and $H\Delta$ contains an open dense subset of $V$. Therefore $H\Delta$ must intersect $V_0$ in a dense subset, and in particular it intersects the set of singular vectors in $V_0$ in a dense subset. This implies that the dense orbit on singular $1$-spaces has a representative in $V_0$, and we are done.

 \end{proof}
 
 Lemma~\ref{dimensionConditions} allows us to reduce the number of cases from Proposition~\ref{simple conclusion big 0-space} to look at, by considering the zero weight spaces. In particular we prove one direction of Proposition~\ref{simple conclusion big 0-space}.
 
 \begin{lemma}\label{adjointCandidates}
Let $(H,V)$ be as in the hypotheses of Proposition~\ref{simple conclusion big 0-space}. If $H$ has a dense orbit on singular $1$-spaces in $V$ and $V$ is not a finite orbit module, then $(H,V)$ is as in the conclusion of Proposition~\ref{simple conclusion big 0-space}.

\end{lemma}
\begin{proof}
We first recall that by the proof of  \cite[Lemma~2.4]{finite}, for all the cases in the hypothesis of Proposition~\ref{simple conclusion big 0-space} $\dim H-\dim C=\dim V-\dim V_0$, where the notation matches Lemma~\ref{dimensionConditions}. Therefore by Lemma~\ref{dimensionConditions} we can assume that $\dim V_0\leq 2$. In fact since all of the cases in Proposition~\ref{simple conclusion big 0-space} where $\dim V_0=1$ are finite orbit modules (see \cite[Lemma 2.4]{finite} and subsequent discussion), we can assume that $\dim V_0=2$.

Going through the list of adjoint modules in Table~\ref{tab:Adjoint modules}, we find that the ones with $\dim V_0=2$ are given by $(A_2,p\neq 3),(A_3,p=2),(B_2,p\neq 2),(D_4,p=2),(G_2,p\neq 3)$. Note that when $p=2$, $V_{A_2}(\lambda_1+\lambda_2)$ and $V_{A_3}(\lambda_1+\lambda_3)$ are orthogonal, by \cite[Thm. 5.1]{mikko}. Also $V_{D_4}(\lambda_2)$ and $V_{G_2}(\lambda_2)$ are orthogonal if $p=2$, by \cite[Thm. 4.2, Prop 6.1]{mikko}.

Finally $V_{C_n}(\lambda_2)$ has a two-dimensional zero weight space only if $n=3,p\neq 3$ or $n=4,p=2$.
\end{proof}

In the upcoming discussion we will show that all of the modules listed in the conclusion of Lemma~\ref{adjointCandidates} are finite singular orbit modules. This will conclude the proof of Proposition~\ref{simple conclusion big 0-space}.

We start by dealing with the $F_4$ cases from the conclusion of Lemma~\ref{adjointCandidates}, which will allow us to conclude also for $H=C_3$.

\subsection{The minimal module for $F_4$}

In this subsection we prove the converse statement in Proposition~\ref{simple conclusion big 0-space} when $H=F_4$ or $H=C_3$.

First note that when $p=2 $ the $26$-dimensional module $V_{F_4}(\lambda_1)$ can be obtained from the minimal module $V_{F_4}(\lambda_4)$, by applying an automorphism of $F_4$. We now proceed to show that the $26$-dimensional minimal module $V_{F_4}(\lambda_4)$ is a finite singular orbit module.

As with $SL_2$, we first look at what happens over finite fields. The orbits of $F_4(q)$ acting on $1$-spaces in the minimal module are given in \cite[\S $B.1$]{F4}. Adopting the same notation set $\epsilon\in\{2,\dots,7\}$ with $q\equiv \epsilon \mod 6$. The orbits are the following:

\begin{center}
 \begin{tabular}{||c c c c ||} 
 \hline
 Orbit type & Number of orbits & Stabilizer & Orbit size \\ [0.5ex] 
 \hline\hline
  I & $ 1 $ & $B_4(q)$ & $q^8 (q^8 + q^4 + 1)$ \\ 
 \hline
 II &$ 1 $& $[q^{15}].B_3(q).(q-1)$ & $(q^{12} - 1) (q^4 + 1)/(q - 1)$  \\
 \hline
 III &$ 1$& $[q^{14}].G_2(q).(q-1)$ & $q^4 (q^{12} - 1) (q^8 - 1)/(q - 1)$  \\
 \hline
 IV &$ 1 $& $[q^7].B_3(q)$ & $q^8 (q^{12} - 1) (q^4 + 1)$  \\
 \hline
 V &$ \frac{q-\epsilon}{6}$& $D_4(q)$ & $q^{12} (q^{12} - 1) (q^8 - 1)/(q^4 - 1)^2$  \\ 
 \hline
  VI &$\delta_{\epsilon,5}+\delta_{\epsilon,7}$& $D_4(q).2$ & $q^{12} (q^{12} - 1) (q^8 - 1)/2(q^4 - 1)^2$  \\ 
 \hline
 VII &$\delta_{\epsilon,4}+\delta_{\epsilon,7}$& $D_4(q).3$ & $q^{12} (q^{12} - 1) (q^8 - 1)/(3 (q^4 - 1)^2)$  \\ 
 \hline
 VIII &$\delta_{\epsilon,3}$& $D_4(q).Sym_3$ & $q^{12} (q^{12} - 1) (q^8 - 1)/(6 (q^4 - 1)^2)$  \\ 
 \hline
 IX &$ \frac{q+1-2\delta_{\epsilon,7}-2\delta_{\epsilon,4}-\delta_{\epsilon,3}}{3}$& $^3D_4(q)$ & $q^{12} (q^8 - 1) (q^4 - 1)$  \\ 
 \hline
 X &$ 2\delta_{\epsilon,7}+2\delta_{\epsilon,4}+\delta_{\epsilon,3}$& $^3D_4(q).3$ & $q^{12} (q^8 - 1) (q^4 - 1)/3$  \\ 
 \hline
 XI &$ \frac{q-1+\delta_{\epsilon,2}+\delta_{\epsilon,4}}{2}$& $^2D_4(q)$ & $q^{12} (q^{12} - 1)$  \\ 
 \hline
  XII &$ \delta_{\epsilon,7}+\delta_{\epsilon,5}+\delta_{\epsilon,3}$& $^2D_4(q).2$ & $q^{12} (q^{12} - 1)/2$  \\ 
 \hline
\end{tabular}
\end{center}
Looking at the orbit sizes we are able to figure out which orbits correspond to singular vectors.
\begin{lemma}\label{F4(q)}
There are at most $5$ orbits of $F_4(q)$ acting on singular $1$-spaces in $V_{F_4(q)}(\lambda_4)$. 
\end{lemma}

\begin{proof}
First recall that the number of singular $1$-spaces is $(q^{12} - 1) (q^{13} + 1)/(q - 1)$ or $(q^{12} + 1) (q^{13} - 1)/(q - 1)$, respectively if the type of the quadratic form is minus or plus. There are only three orbit types where the number of orbits is not bounded, types $V,IX,XI$. Orbit type $IX$ is the one with the smallest size among types $V,IX,XI$. An orbit of type $IX$ has size $N=q^{12} (q^8 - 1) (q^4 - 1)$ and it is easy to see that $10N> max\{ (q^{12} - 1) (q^{13} + 1)/(q - 1),(q^{12} + 1) (q^{13} - 1)/(q - 1)\}$. This means that there can only be at most $10$ orbits for each type that are singular orbits.

We can explicitly find which orbits are singular by computing in which case the sizes of the orbits add up to the number of singular vectors.
If $q\equiv 4 \mod 6$ we find that there are $5$ singular orbits, of type $II,III,VII,X$, with type $X$ occurring twice. Otherwise there are $3$ singular orbits, of types $II,III,XI$. Note that in the first case the form is of $'+'$ type and in the second of type $'-'$.
\end{proof}

We are therefore able to conclude in the algebraically closed case.
\begin{corollary}\label{F4}
$F_4$ has $3$ orbits on singular $1$-spaces in $V=V(\lambda_4)$. The generic stabilizer is $D_4.3$.
\end{corollary}
\begin{proof}
It is clear by looking at the stabilizers of the singular orbits of type $II,III,VII,X$ and $II,III,XI$ for $F_4(q)$, that the stabilizers of the orbits in the algebraic case are $U_{15}B_3T_1$, $U_{14}G_2T_1$ and $D_4.3$, with the orbit with stabilizer $D_4.3$ splitting over finite fields according to Lang-Steinberg.
\end{proof}

Finally we can use the results about $F_4$ to conclude for $C_3$. We follow the strategy of \cite[Lemma 2.12]{finite}

\begin{lemma}\label{c3}
If $H=C_3$, $p\neq 3$, then $V=V_H(\lambda_2)$ is a finite singular orbit module.  A generic stabilizer is $A_1^3.(2^3.3)$ if $p=2$.
\end{lemma}

\begin{proof}
Let $Y=F_4$, and let $M$ be the $26$-dimensional $Y$-module $V_Y(\lambda_4)$. If $A_1$ denotes a fundamental $SL_2$ in $Y$, then $N_Y(A_1)=A_1C_3$, and $$M\downarrow A_1C_3=(V_{A_1}(\lambda_1)\otimes V_{C_3}(\lambda_1))\oplus V_{C_3}(\lambda_2)$$ and we can take $V=C_M(A_1)$ (see \cite[Lemma 2.12]{finite}).
By the proof of \cite[Lemma 2.12]{finite} if two $1$-spaces in $V$ are $Y$-conjugate then they are conjugate under $N_Y(A_1)$. But since $Y$ has finitely many orbits on singular $1$-spaces in $M$ by Corollary~\ref{F4} we are done. 
Considering an explicit description of $V$ as constructed for $V_{C_4}(\lambda_2)$ in the upcoming subsection, a generic stabilizer if $p=2$ is $A_1^3.(2^3.3)$. This follows in the same way as in Corollary~\ref{c4d4}.
\end{proof}

\subsection{Adjoint cases}

In this section we conclude the proof the converse statement in Proposition~\ref{simple conclusion big 0-space}. In the previous section we have dealt with $F_4,C_3$. We now show that all the remaining cases in the conclusion of Proposition~\ref{simple conclusion big 0-space} are finite orbit singular modules.

We start by dealing with $A_3$ case. Let $H=A_3=SL_{4}(k)$, for an algebraically closed field $k$ of characteristic $2$. The module $V=V_{A_3}(\lambda_1+\lambda_3)$ corresponds to the Lie algebra $sl(4,k)$ of trace $0$ matrices quotiented by the scalar matrices (note that they have trace $0$ if $p=2$), on which $H$ acts by conjugation.

Let $e_{ij}$ be the $4\times 4$ matrix with $0$ everywhere apart from a $1$ in position $(i,j)$. Let $\overline{e_{ij}} \in V $ be $e_{ij}+\langle I\rangle$. We describe the quadratic form $Q$ stabilized by $H$ with the following lemma.

\begin{lemma}\label{formA3}
For $i\neq j$ the elements $\overline{e_{ij}}$ and $\overline{e_{ji}}$ form hyperbolic pairs. Furthermore the two singular $1$-spaces in the zero weight space $V_0$ have representatives $diag(0,1,a,1+a)+ \langle I \rangle$ and $diag(0,1,b,1+b)+ \langle I \rangle$, where $a,b$ are the two roots of $1+x+x^2$.
\end{lemma}

\begin{proof}
We start by observing that $V_0$ is $2$-dimensional, consists of trace $0$ diagonal matrices quotiented by $\langle I \rangle$, and is perpendicular to all other weight spaces.
A highest weight vector is $v^+=\overline{e_{14}}$, since $\langle \overline{e_{14}} \rangle$ is fixed by the Borel subgroup of upper triangular matrices. Since $V\simeq V^*$, \cite[Prop. 16.1]{MT} shows that $(v^+,v^-)=1$, where $v^-=\overline{e_{41}}$. Therefore $v^+$ and $v^-$ form a hyperbolic pair. Now simply acting with the Weyl group $W=Sym_4$ we find all the other hyperbolic pairs $(\overline{e_{ij}},\overline{e_{ji}})$ with $i<j$.

Let $v_0=diag(0,a,a,0)+ \langle I \rangle\in V_0$, with $a\neq 0$. If $v_0$ is singular then acting by conjugation with the Weyl group we find that the elements $u_0=diag(0,a,0,a)+ \langle I \rangle$ and $u_0+v_0=diag(0,0,a,a)+\langle I \rangle$ are also singular. But then  $(u_0,v_0)=Q(u_0)+Q(v_0)+Q(u_0+v_0)=0$, making the form degenerate on $V_0$, which is absurd. Therefore $Q(u_0)=Q(v_0)=(u_0,v_0)\neq 0$. Let $x=u_0+\lambda v_0$ be a singular element of $V_0$. Then $Q(x)=Q(u_0+\lambda v_0)=Q(u_0)+\lambda^2Q(u_0)+\lambda Q(u_0)$ and we must have $1+\lambda+\lambda^2=0$.
\end{proof}

With the given description of the quadratic form on $V_{A_3}(\lambda_1+\lambda_3)$, we are able to conclude.

\begin{lemma}\label{a3}
The group $A_3$ has finitely many orbits on the singular $1$-spaces in its $14$-dimensional module $V(\lambda_1+\lambda_3)$ with $p=2$. The stabilizer of a point in the dense orbit is  $T_3.Alt_4$.
\end{lemma}

\begin{proof}
Let $M\in sl_4$ such that $M+\langle I\rangle $ is singular. If $M$ can be diagonalized it is in the same orbit as one of the two singular $1$-spaces in $V_0$. Otherwise there must be an eigenvalue with multiplicity bigger than $1$. We can conjugate $M$ to an element $M'$ in Jordan Canonical Form. Since $M'$ is upper triangular, it is a sum of a strictly upper triangular matrix and an element $D_0$ in $V_0$. Since $V_0$ is perpendicular to the non-zero weight spaces, in order for $M'+\langle I \rangle$ to be singular, $D_0$ must be singular, hence $D_0=0$ by Lemma~\ref{formA3}. This means that $M+\langle I \rangle$ can be represented by a a nilpotent matrix. But $SL_4$ acts with finitely many orbits on $4\times 4$ nilpotent matrices. Therefore, as claimed, $V_{A_3}(\lambda_1+\lambda_3)$ is a finite singular orbit module.

Take a singular element element $ x:=diag(0,1,a,1+a)+\langle I \rangle\in V_0$, with $a^2+a+1$. We show that $x$ is the representative of a singular $1$-space with stabilizer $T.Alt_4$. To do this let $h\in SL_4$ be such that $h^{-1} diag(0,1,a,1+a) h =\lambda diag(0,1,a,1+a) +\alpha I$ for some $\lambda \in k^*,\alpha\in k$. Since conjugation preserves the spectrum, $\lambda diag(0,1,a,1+a) +\alpha I$ must be a diagonal matrix with entries $0,1,a,1+a$. Since $1+a+a^2=0$ this happens if and only if $\alpha, \lambda\in \{0,1,a,1+a\}$ with $\lambda\neq 0$, inducing an $Alt_4$ action on $(0,1,a,a+1)$.

This means that the stabilizer of the $1$-space spanned by $x$ is given by the centralizer $T$ of  $diag(0,1,a,1+a)$ in $SL_4$ extended by $Alt_4$.
\end{proof}

We now deal with the cases $H=D_4,C_4$ with $p=2$ and $V=V(\lambda_2)$. Note that $V_{D_4}(\lambda_2)=V_{C_4}(\lambda_2)\downarrow D_4$. We start by giving an explicit description of the $26$ dimensional module $V_{D_4}(\lambda_2)$ in terms of the Lie algebra $so_8$. 

The Lie algebra $so_8$ with $p=2$ consists of the set of $8\times 8$ matrices of the form $\left(\begin{matrix} A & P\\ Q & A^T \\ \end{matrix}\right) $, where $A,P,Q$ are $4\times 4$ matrices with $P=P^T$ and $Q=Q^T$ and with $P,Q$ having zero diagonals. Let $$W=\lbrace \left(\begin{matrix} A & P\\ Q & A^T \\ \end{matrix}\right)\in so_8 : tr(A)=0\rbrace .$$ The subspace $W$ is stable under conjugation by $D_4$ and $I\in  W$ is fixed under this action. The $26$-dimensional module $V=V_{D_4}(\lambda_2)$ is then obtained by taking the quotient of $W$ by $\langle I \rangle$. 

Let $e_{ij}$ be the $8\times 8$ matrix with $0$ everywhere apart from a $1$ in position $(i,j)$. Let $\overline{e_{ij}}$ be $e_{ij}+\langle I\rangle$. 

\begin{lemma}\label{form c4 d4 adjoin}
Let $\overline{v}=v+\langle I \rangle$ for $v\in W$ be a singular vector in $V$. Then $\overline{v}$ is either semisimple or nilpotent.
\end{lemma}

\begin{proof}
Assume that there is a singular element $v=s+n$ for a semisimple element $s$ and a nilpotent element $n$ in $so_8$, with $[s,n]=0$ and $s,n\notin \langle I\rangle$. Since we can find an element $h\in H=D_4$ sending $s+\langle
 I \rangle $ to $V_0$, we can assume that $s$ is a diagonal matrix. By \cite[Table 8.5b]{unipotent} there is a set of nilpotent representatives with zero diagonal for the action of $D_4$ on nilpotent elements in $so_8$. Therefore we actually have $n,s\in W$. 
  We have two cases: either $s$ has $4$ distinct diagonal entries or not.

In the first case $s$ is of the form $diag(a,b,c,a+b+c,a,b,c,a+b+c)$, with $a,b,c,a+b+c$ distinct. It is easy to see that because the diagonals of the $P$ and $Q$ blocks of elements of $W$ are $0$, the centralizer of $v$ in $W$ consists of the subspace of diagonal matrices, which is absurd since $n$ is nilpotent and not a scalar matrix.

In the second case $s$ is of the form  $diag(a,a,b,b,a,a,b,b)$ with $a\neq b$. The centralizer in $W$ of $s$ is the set of matrices $$\left\lbrace  \left(\begin{matrix} A_1 & 0 & P_1 &0\\ 0 & A_2 & 0 &P_2 \\Q_1 & 0 & A_1^T &0\\ 0 & Q_2 & 0 &A_2^T \\ \end{matrix}\right)    \right\rbrace,$$ where $A_1,A_2,P_1,P_2,Q_1,Q_2$ are $2\times 2$ matrices with $tr(A_1)+tr(A_2)=0$ and $P_1,P_2,Q_1,Q_2$ antidiagonal. 

The nilpotent element $n$ must be such a matrix. Computation on powers of $n$ shows that there are two cases: 
\begin{enumerate}[label=(\roman*)]
\item $tr(A_1)=0$ and one of $(P_1,P_2)$ is $0$ and so is one of $(Q_1,Q_2)$;
\item $P_1=P_2=Q_1=Q_2=0$.
\end{enumerate}

To conclude we require a couple of considerations about the quadratic form on $V$. Starting from a highest weight vector $v^+=\overline{e_{14}}+\overline{e_{85}}$ we get the hyperbolic pair $(v^+,v^-)$, where $v^-=\overline{e_{41}}+\overline{e_{58}}$. Acting with the Weyl group we get hyperbolic pairs $(\overline{e_{16}}+\overline{e_{25}},\overline{e_{52}}+\overline{e_{61}})$ and $(\overline{e_{38}}+\overline{e_{47}},\overline{e_{74}}+\overline{e_{83}})$. Furthermore as in Lemma~\ref{formA3} the element $s+\langle I \rangle $ is non-singular. 

We now want to show that we can reduce case $(i)$ to case $(ii)$.
By the above considerations on the quadratic form, in case $(i)$ the matrix $n_{pq}:= \left(\begin{matrix} 0 & 0 & P_1 &0\\ 0 & 0 & 0 &P_2 \\Q_1 & 0 & 0 &0\\ 0 & Q_2 & 0 & 0 \\ \end{matrix}\right)    $ is singular. It also commutes with the matrix $n_a:=n+n_{pq}$, and is nilpotent. Therefore we find a singular element $\overline{u}=s+n_a+\langle I \rangle$ with $[s,n_a]=0$, as in case $(ii)$. Since $s +\langle I \rangle$ is not singular, this is not possible by our proof of Lemma~\ref{a3}.
\end{proof}

\begin{corollary}\label{c4d4}
If $p=2$ then $V=V_{D_4}(\lambda_2)$ and $V=V_{C_4}(\lambda_2)$ are finite singular orbit modules for $D_4$ and $C_4$ respectively. Let $v=diag(0,1,s,1+s,0,1,s,1+s)+\langle I \rangle$ be a singular semisimple vector in $V_0$, so that $1+s+s^2=0$. Then $(D_4)_{\langle v \rangle }=T_4.(2^3.Alt_4)$ and $(C_4)_{\langle v \rangle }=(A_1)^4.(2^4.Alt_4)$. 
\end{corollary}

\begin{proof}
By Lemma~\ref{form c4 d4 adjoin} singular elements in $V$ are either semisimple or nilpotent. Since there is only one orbit on semisimple elements and finitely many on nilpotent elements we are done. The stabilizer of $diag(0,1,s,1+s,0,1,s,1+s)$ in $C_4$ is the set of matrices $\{ a_{ij}\in C_4: a_{ij} \neq 0 \leftrightarrow (i=j \vee i+4=j \vee i=j+4) \}$, which is easily seen to be $(A_1)^4 \leq C_4$. At the same time $(A_1)^4\cap D_4 = T_4$, This shows that $(D_4)_{\langle v \rangle }^0=T_4$ and $(C_4)_{\langle v \rangle }^0=(A_1)^4$. 

Consider the Weyl groups $W(C_4)=2^4.Sym_4$ and $W(D_4)=2^3.Sym_4$. The subgroups $2^4$ and $2^3$ clearly fix $v$. As in the proof of Lemma~\ref{a3} we find that $(Sym_4)_{\langle v \rangle}=Alt_4$ and our result is proved.
\end{proof}

We now deal with the remaining cases, i.e. the adjoint modules for $(A_2,p\neq 3)$, $(B_2,p\neq 2)$ and $(G_2,p\neq 3)$. Regard any such $H$ as a subgroup of $SO(V)$ and note that $V=Lie(H)$. The strategy of the proof will be again to show that there every singular element is semisimple or nilpotent.

\begin{lemma}\label{adjoint ss nilpotent}
Let $H$ be one of $(B_2,p\neq 2)$, $(A_2,p\neq 3)$ and $(G_2,p\neq 3)$. Let $V=Lie(H)$ with $H$ stabilizing a non-degenerate quadratic form $Q$ on $V$, let $s$ be a non-zero singular semisimple element of $V$ and let $n$ be a non-zero nilpotent element of $V$. Then $[s,n]\neq 0$ and $n$ is singular.
\end{lemma}

\begin{proof}
Note that the zero-weight space $V_0$ of $V$ is two-dimensional and any semisimple element can be conjugated by $H$ into $V_0$. We can therefore assume that $s\in V_0$.

 If $H=B_2$ or $H=A_2$ then a non-degenerate orthogonal form fixed by $H$ is $(x,y)=tr(xy)$ for any $x,y\in Lie(H)$. If $p\neq 2$ in both cases nilpotent elements are clearly singular. If $p=2$ and $H=A_2$ simply note that nilpotent elements can be conjugated to strictly upper triangular matrices, that are singular as in the $A_3$ case (see Lemma~\ref{formA3}).  
 
If $H=B_2$ then $s$ must be of the form $diag(0,a,b,-a,-b)$ and since it is singular we must have $a^2+ b^2=0$, giving that all the entries $0,a,b,-a,-b$ are distinct. The centralizer of such an element must be diagonal and therefore an element of $V_0$. Similarly if $H=A_2$ and $p\neq 2$ then $s$ must be of the form $diag(a,b,-a-b)$ and since it is singular we must have $a^2+b^2+ab=0$, giving that all the entries $a,b,-a-b$ are distinct. If $H=A_2$ and $p=2$, as in Lemma~\ref{formA3}, $diag(a,a,0)$ cannot be singular, giving again that all the entries $a,b,a+b$ must be distinct. The centralizer of $s$ is then just an element of $V_0$.
 
To conclude assume that $H=G_2$. We realize the Lie algebra $\mathfrak{g}_2$ as a subalgebra of $\mathfrak{so}_7$. Here $V_0$ can be seen as the set of diagonal matrices $\{diag(0,a,b,-a-b,-a,-b,a+b):a,b\in k\}$ ( see \cite[19.3]{HumRep}). Assume for a contradiction that $Q(diag(0,a,a,-2a,-a,-a,2a))=0$. Then by acting with the Weyl group $W(G_2)=Dih_{12}$ we see that $u:=diag(0,a,-2a,a,-a,2a,-a)$ and $u+v=diag(0,2a,-a,-a,-2a,a,a)$ are also singular. But this gives that $(u,v)=Q(u)+Q(v)+Q(u+v)=0$ which is absurd since the form must be non-degenerate on $V_0$. The same type of argument shows that $a,b\neq 0$ and $a\neq -b$ for any singular element of $V_0$. 

When $p\neq 2$ all the entries on the diagonal of $s$ are therefore distinct and the centralizer of $s$ is $V_0$. Let $M_7$ be the vector space of $7\times 7$ matrices over $k$. When $p=2$ we find that the centralizer of $s$ in $M_7$ is spanned by diagonal matrices together with $e_{52},e_{63},e_{74},e_{25},e_{36},e_{47}$.
However $\mathfrak{g}_2$ is contained in the subspace of $M_7$ spanned by the standard basis without the basis vectors $e_{52},e_{63},e_{74},e_{25},e_{36},e_{47}$ (see description of roots in \cite[19.3]{HumRep}). This shows that the centralizer in $\mathfrak{g}_2$ of $s$ is $V_0$ also if $p=2$.

It remains to show that nilpotent elements in $\mathfrak{g}_2$ are singular. If $p\neq 2$ we can take the Killing form $\kappa$ as the orthogonal form on $V$. We then have $(n,n)=\kappa(n,n)=Tr(ad(n)\circ ad(n))$ and since $n$ is nilpotent so is $ad(n)$, which gives the result.
If $p=2$ we can refer to the explicit construction for the $G_2$-invariant quadratic form in \cite[\S 5]{Babi}. Here the quadratic form is constructed by first computing the quadratic form in a Chevalley basis of the Lie algebra of a $G_2$ defined over $\mathbb{Z}$ and in particular it  tells us that if the Killing form is $0$ in a Chevalley basis, then so is the $G_2$-invariant quadratic form $Q$ when $p=2$. 
\end{proof}

Thanks to the previous Lemma we can now conclude.

\begin{proposition}\label{B2 A2 G2}
Let $H$ be one of $(B_2,p\neq 2)$, $(A_2,p\neq 3)$ and $(G_2,p\neq 3)$. Let $V=Lie(H)$. Then $V$ is a finite singular orbit module. The generic stabilizers are respectively $T_2.4$, $T_2 (p\neq 2)$ or $T_2.3 (p=2)$, $T_2 (p\neq 2)$ or $T_2.Dih_{6} (p=2)$.
\end{proposition}

\begin{proof}
Let $v$ be a singular element in $V$. Write $v=s+n$ for a semisimple element $s$ and nilpotent $n$ with $[s,n]=0$. If $H=B_2$ or $A_2$, since $s$ and $n$ commute and $n$ is nilpotent, $sn$ is nilpotent and therefore $(s,n)=Tr(sn)=0$. Similarly if $H=G_2$ with $p\neq 2$, $(s,n)=Tr(ad(s)\circ ad(n))=0$. Finally if $H=G_2,p=2$ the same argument from the proof of Lemma~\ref{adjoint ss nilpotent} using \cite[\S 5]{Babi} gives us $(s,n)=0$. 
Therefore since $n$ is singular by Lemma~\ref{adjoint ss nilpotent}, $s$ is also singular. But by Lemma~\ref{adjoint ss nilpotent} this means that either $s=0$ or $n=0$. There are only two singular $1$-spaces in $V_0$ and finitely many orbits on nilpotent elements. 

We now explicitly compute the generic stabilizers. For $B_2$ we are looking for the centralizer of $\langle v \rangle=\langle diag(0,a,b,-a,-b) \rangle$ where $a^2=-b^2$. The connected component of the centralizer is $T_2$, while there is a cyclic subgroup of order $4$ of the Weyl group of $B_2$ centralizing $\langle v \rangle$. For $A_2$ a generic stabilizer is the centralizer of $\langle v \rangle =\langle diag(1,s,-1-s)$ where $1+s^2=s$. We get $(A_2)_{\langle v \rangle}^0=T_2$. If $p\neq 2$ there is no element in $W(A_2)=Sym_3$ stabilizing $\langle v \rangle$, while if $p=2$ then $(Sym_3)_{\langle v \rangle}=Alt_3$.
Finally for $G_2$ let $\langle v \rangle = \langle diag(0,1,s,-1-s,-1,-s,1+s) \rangle$ where $1+s^2=s$. Computation shows that $(G_2)_{\langle v \rangle}^0=T_2$. If $p\neq 2$ then no element of the Weyl group $Dih_6$ stabilizes $\langle v \rangle$, while if $p=2$ the full $Dih_6$ does.
\end{proof}

Note that this concludes the proof of Proposition~\ref{simple conclusion big 0-space}, putting together Lemma~\ref{adjointCandidates}, Corollary~\ref{F4}, Corollary~\ref{c4d4}, Lemma~\ref{c3}, Lemma~\ref{a3} and Proposition~\ref{B2 A2 G2}.

This means that have finally proved Theorem~\ref{main theorem}. 

To end the section, note that Proposition~\ref{simple conclusion big 0-space} gives us the proof of Corollary~\ref{main corollary}: for if $H\leq SO(V)$ is simple and has a dense orbit on singular $1$-spaces in $V$, where $V$ is a faithful irreducible rational $H$-module, then $V$ is in the conclusion of Theorem~\ref{simple candidates}. For each candidate given by Theorem~\ref{simple candidates}, we have shown that either $V$ is a finite singular orbit module, or $H$ has no dense orbit on singular $1$-spaces in $V$, with the latter statement being a direct consequence of Proposition~\ref{simple conclusion big 0-space}.

\section{The semisimple case: tensor decompositions}\label{tensor section}

In this section we prove Theorem~\ref{main theorem 2}. 

By Proposition~\ref{subgroupstructure} the only maximal connected subgroups of $SO(V)$ that preserve a tensor decomposition are of the form $SO(V_1)\otimes SO(V_2)$ (if $p\neq 2$) or $Sp(V_1)\otimes Sp(V_2)$, where $V=V_1\otimes V_2$. Note that if $p=2$ then $SO(V_1)\otimes SO(V_2)$ is not maximal in $SO(V_1\otimes V_2)$. We are not going to consider the cases with $SO_2$ as one of the factors of $H$, since this is just a $1$-dimensional $T_1$. 

We begin with some general facts.
Denote by $Cl_n$ a simple classical algebraic group with natural module $W_n$ of dimension $n$ over $k$.
The following lemma provides a useful reduction.
\begin{lemma}\cite[Lemma 4.2]{finite}\label{reductionTensor}
Suppose that $G=G_1\otimes Cl_n$ acting on $V=W_m\otimes W_n$, for $G_1\leq GL_m$. Let $U$ be a non-degenerate subspace of $W_n$ (any subspace if $Cl_n=SL_n$), and let $S$ be the stabilizer of $U$ in $Cl_n$. Two $k$-dimensional subspaces of $W_m\otimes U$ are in the same $G$-orbit if and only if they are in the same $G_1\otimes S$ orbit.
\end{lemma}

We can use this with dimensional considerations in order to exclude some possibilities.
\begin{corollary}\label{two classical 4}
Let $H=Cl(V_1)\otimes Cl(V_2)=H_1\otimes H_2<SO(V)$ be a maximal semisimple group acting irreducibly on $V=V_1\otimes V_2$. If $V$ is a finite singular orbit module but not a finite orbit module then $H$ is of the form $Sp_m\otimes Sp_n$, with $m,n\geq 4$.	
\end{corollary}

\begin{proof}
Without loss of generality assume that $\dim V_1\leq \dim V_2$. If $\dim V_1=2$ then $H_1=Sp_2=SL_2$ and we have a finite orbit module by \cite[Table $1$]{finite}. If $\dim V_1=3$ then $H_1=SO_3$ and $p\neq 2$, and so $H_2=SO_n$, with $n=\dim V_2$. To prove the statement it is therefore sufficient to prove that $H=SO_m\otimes SO_n$ ($p\neq 2$) does not have finitely many orbits on singular $1$-spaces. Let $\dim V_1=m$ and $\dim V_2=n$ and consider $H=SO_m\otimes SO_n$, with $p\neq 2$ and $m,n\geq 3$.

Suppose that $V$ is a finite singular orbit module and let $U$ be a non-degenerate $m$-dimensional subspace of $V_2$. The group induced on $U$ by the stabilizer $S$ of $U$ in $H$ is $SO_m$ and by Lemma~\ref{reductionTensor} $SO_m\otimes SO_m$ must have finitely many orbits on singular $1$-spaces in $V_1\otimes U$. However $\dim SO_mSO_m=m^2-m$ and we need $\dim SO_mSO_m\geq m^2-2$. This is impossible if $m\geq 3$.
\end{proof}

Note that considerations about the dimension never rule out the $Sp_n\otimes Sp_m$ cases.

Our strategy will now consist of showing that $Sp_6\otimes Sp_6$ has infinitely many orbits on singular $1$-spaces in $W_6\otimes W_6$, and that the opposite is true for $Sp_4\otimes Sp_4$ acting on singular $1$-spaces in $W_4\otimes W_4$. We now describe the setting in which we will be able to deal with both of these cases. 

We identify the vector space $V=W_n\otimes W_n$ with the vector space $M_n$ of $n\times n$ matrices over $k$. The natural action of $GL_n\otimes GL_n$ on $V$ is then given by $A\cdot g_1\otimes g_2=g_1Ag_2^T$ for $g_1,g_2\in GL_n$. Given this setting we will refer to the rank of a vector $v\in V$, which is simply the rank of the $n\times n$ matrix corresponding to $v$. The quadratic form $Q$ preserved by $Sp_{2n}\otimes Sp_{2n}$ on $M_{2n}$ is then given by:

$$Q(\textbf{a}_{ij})=\sum_{1\leq i,j\leq n}a_{ij}a_{2n+1-i,2n+1-j}-a_{i,j+n}a_{2n+1-i,n+1-j}.$$

It will be useful to consider the $(Sp_n,Sp_n)$-double cosets in $SL_n$. We follow \cite[Prop. 4.1]{finite}. Let $\Gamma=SL_n(k)$ and $H=Sp_n(k)$. Let $\tau$ be the automorphism of $\Gamma$ sending $g\in \Gamma$ to $J^{-1}g^{-T}J$ for $J=\left(\begin{matrix}0 & -I_{n/2} \\ I_{n/2} & 0 \\ \end{matrix}\right)$. Then $\Gamma_\tau=H$ and the $(H,H)$-double cosets in $G$ are in bijection with the orbits of $H$ acting by conjugation on $\{\tau(g^{-1})g:g\in \Gamma\}$. This bijection is obtained by sending a double coset $HgH$ to the orbit with representative $\tau(g^{-1})g$.

\subsection{$Sp_6\otimes Sp_6$}
Here we show the following:

\begin{proposition}\label{sp6sp6 main}
The semisimple group $Sp_m\otimes Sp_n$ with $m,n\geq 6$ has infinitely many orbits on singular $1$-spaces in $V=W_m\otimes W_n$.
\end{proposition}

We begin by showing that $Sp_6\otimes Sp_6$ has infinitely many orbits on singular $1$-spaces in $V=M_6$. Let $\Gamma=SL_6$ and $H=Sp_6$, with $H$ acting by conjugation on $\{\tau(g^{-1})g:g\in \Gamma\}$ as described above.

Since $Sp_6$ is closed under transposition and $\Gamma$ is a subset of $V$, any left or right multiplication of $A\in M_6$ by $g\in H$ preserves the given quadratic form, and therefore any $H\backslash \Gamma \slash H$ double coset has constant quadratic form on its elements. We say that a double coset is singular if a representative is singular.
With the above settings we have the following lemma.
\begin{lemma}\label{double cosets sp6sp6}
There exist infinitely many $H\backslash \Gamma \slash H$ singular double cosets.
\end{lemma}
\begin{proof}
Consider the set $\mathcal{A}$ of singular diagonal matrices in $\Gamma$. A general $A\in\mathcal{A}$ is of the form $$A=diag(\alpha,\beta,\gamma,\delta,\epsilon,\theta)$$ with $\alpha= (\beta\gamma\delta\epsilon\theta)^{-1}$ and $\alpha\theta+\beta\epsilon+\gamma\delta=0$. Set $a=\alpha\theta,b=\beta\epsilon,c=\gamma\delta$ and note that $$\tau(A^{-1})A =diag\left(\frac{1}{bc},b, c , c , b ,\frac{1}{bc} \right)$$ under the condition $1+b^2c+bc^2=0$. The set of eigenvalues of $\tau(A^{-1})A$ is $\{\frac{1}{bc},b,c\}$. Therefore, since conjugation preserves the spectrum, and  $b$ can be any non-zero element of $k$, we must have infinitely many orbits of $H$ on $\{\tau(A^{-1})A: A\in \mathcal{A}\}$. We therefore have infinitely many double cosets with representatives in $\mathcal{A}$. 
\end{proof}
We are now able to conclude.
\begin{corollary}\label{Sp_6}
The semisimple group $Sp_6\otimes Sp_6$ has infinitely many orbits on singular $1$-spaces in $V=W_6\otimes W_6$. 
\end{corollary}
\begin{proof}
Lemma~\ref{double cosets sp6sp6} shows that we can find an infinite list $(Sp_6g_iSp_6)_i$ of distinct singular $(Sp_6,Sp_6)$-double cosets in $SL_6$, where we can assume that no $g_i$ is a scalar multiple of another $g_j$ for $i\neq j$. 
The $Sp_6Sp_6$-orbit with representative $g_i$ consists precisely of all the elements in  $Sp_6g_iSp_6$. The $Sp_6Sp_6$-orbit of the $1$-space $\langle g_i \rangle$ consists of all the elements in $Sp_6\langle g_i \rangle Sp_6$. By construction each $\langle g_i \rangle$ therefore lies in a different $Sp_6Sp_6$-orbit and we are done. 
\end{proof}

We can finally conclude the proof of Proposition~\ref{sp6sp6 main}. This follows from Lemma~\ref{reductionTensor} together with Corollary~\ref{Sp_6}.

\subsection{$Sp_4\otimes Sp_4$}
In this section we conclude the proof of Theorem~\ref{main theorem 2}. By Proposition~\ref{sp6sp6 main} all that remains to be proved is the following:

\begin{proposition}\label{semisimple all finite cases}
Let $n\geq 4$. The group $Sp_4\otimes Sp_n$ has finitely many orbits on singular $1$-spaces in $V=W_4\otimes W_n$. The generic stabilizer is $(A_1A_1).2(Sp_{n-4})$ if $p\neq 2$ and $(U_3A_1)(Sp_{n-4})$ when $p=2$, where we regard $Sp_0$ as the trivial group.
\end{proposition}

We first prove the result with $n=4$ and then extend to the general case. We in fact provide explicit stabilizers for the action of $Sp_4\otimes Sp_4$ on singular $1$-spaces in $V=W_4\otimes W_4$

As mentioned before we take $V$ to be the set $M_4$ of $4\times 4 $ matrices over $k$ and let $\Gamma=SL_4,H=Sp_4$, with $H$ acting by conjugation on $\{\tau(g^{-1})g:g\in \Gamma\}$. We denote by $\{e_1,f_1,e_2,f_2\}$ a standard symplectic or orthogonal basis for the natural module $W_4$ for $Sp_4$ or $SO_4$. For brevity we are going to denote by $v_I$ the tensor corresponding to $I$, i.e. $e_1\otimes e_1+e_2\otimes e_2+ f_2\otimes f_2 +f_1\otimes f_1$.

We now state the explicit results when $n=4$.

\begin{proposition}\label{sp4sp4}
The orbits on singular $1$-spaces in $V=W_4\otimes W_4$ under the action of $H=Sp_4\otimes Sp_4$ are as in Table~\ref{tab:sp4sp4 p not 2} if $p\neq 2$. When $H=Sp_4\otimes Sp_4$ are as in Table~\ref{tab:sp4sp4 p=2}.
\begin{table}[h]
\begin{center}
 \begin{tabular}{||c c ||} 
 \hline
 $H$-orbit rep $\langle v \rangle $&  $H_{\langle v \rangle }$ \\ [0.5ex] 
 \hline\hline
   $e_1\otimes e_1+e_2\otimes e_2+ wf_2\otimes f_2 -wf_1\otimes f_1$ with $w=\sqrt{-1}$ & $A_1A_1.2 $  \\
 \hline
 $e_1\otimes e_1+f_1\otimes e_2+ e_2\otimes f_2 $   &$ U_5.T_2.2 $  \\ 
 \hline
   $e_1\otimes e_1+e_2\otimes e_2$  &   $U_6.(A_1T_2) $ \\
 \hline
   $e_1\otimes e_1+f_1\otimes e_2$  &   $U_3.(A_1A_1T_1) $ \\
\hline
   $e_1\otimes e_1+e_2\otimes f_1$  &   $U_3.(A_1A_1T_1) $ \\
 \hline
  $e_1\otimes e_1$  &   $P_1P_1$ \\
 \hline
 
\end{tabular}
\end{center}
\caption{$Sp_4Sp_4$-orbits when $p\neq 2$}
\label{tab:sp4sp4 p not 2}
\end{table}

\begin{table}[h]
\begin{center}
 \begin{tabular}{||c c  ||} 
 \hline
 $H$-orbit rep $\langle v \rangle $&  $H_{\langle v \rangle }$ \\ [0.5ex] 
 \hline\hline
 $v_I$ & $Sp_4$  \\
 \hline
   $v_I+e_1\otimes e_2$ & $U_3.A_1 $   \\
 \hline
 $e_1\otimes e_1+f_1\otimes e_2+ e_2\otimes f_2 $   &$ U_5.T_2 $   \\ 

 \hline
   $e_1\otimes e_1+e_2\otimes e_2$  &   $U_6.(A_1T_2) $  \\
 \hline
   $e_1\otimes e_1+f_1\otimes e_2$  &   $U_3.(A_1A_1T_1) $   \\
\hline
   $e_1\otimes e_1+e_2\otimes f_1$  &   $U_3.(A_1A_1T_1) $   \\
 \hline
  $e_1\otimes e_1$  &   $P_1P_1$   \\
 \hline
 
\end{tabular}
\end{center} 
\caption{$Sp_4Sp_4$-orbits when $p=2$}
\label{tab:sp4sp4 p=2}
\end{table}

\end{proposition}

\begin{proof}
We divide the proof in terms of the rank of the orbit representatives.

Let $X$ be the set of singular matrices in $M_4$ of determinant $1$. We start by showing that $H=Sp_4\otimes Sp_4$ has $1$ orbit on $X$ if $p\neq 2$ and $2$ orbits otherwise. As previously discussed  we have a bijection between the $Sp_4\backslash SL_4\slash Sp_4$ double cosets and the orbits of $Sp_4$ acting by conjugation on the set $\{\tau(g^{-1})g:g\in SL_4\}$. To prove our claim it is therefore sufficient to consider $Sp_4$ acting by conjugation on $\{\tau(g^{-1})g:Q(g)=0,g\in SL_4\}$.
We start with a singular matrix of determinant $1$, namely 
$$M_1=\left(\begin{matrix}a & b & c & d \\ e & f & g & h \\ i & l & m & n \\ o & p & q & r \\ \end{matrix}\right).$$
We have $\tau(M_1^{-1})=J^{-1}M_1^TJ$ for $J=\left(\begin{matrix}0 & -I_{2} \\ I_{2} & 0 \\ \end{matrix}\right)$. Therefore $$\tau(M_1^{-1})= \left(\begin{matrix}r & n & -h & -d \\ q & m & -g & -c \\ -p & -l & f & b \\ -o & -i & e & a \\ \end{matrix}\right).$$
Calculating $Y_1:=\tau(M_1^{-1})M_1$ we get
$$Y_1= \left(\begin{matrix}-hi+en-do+ar & -hl+fn-dp+br & -hm+gn-dq+cr & 0 \\ -gi+em-co+aq & -gl+fm-cp+bq & 0 & hm-gn+dq-cr \\ fi-el+bo-ap & 0 & -gl+fm-cp+bq & -hl+fn-dp+br \\ 0 & -fi+el-bo+ap & -gi+em-co+aq & -hi+en-do+ar \\ \end{matrix}\right).$$
For clarity we write $Y_1$ in the form 
$$Y_1=\left(\begin{matrix}A & C & D & 0 \\ E & B & 0 & -D \\ F & 0 & B & C \\ 0 & -F & E & A \\ \end{matrix}\right)$$
and computation shows that $-AB+CE+DF=-det(M_1)=-1$.
Now we simplify further using the fact that $M_1$ is singular, to compute that $-hi+en-do+ar-gl+fm-cp+bq=0$ which implies $B=-A$. Therefore $Y_1$ is of the form 
$$Y_1=\left(\begin{matrix}A & C & D & 0 \\ E & -A & 0 & -D \\ F & 0 & -A & C \\ 0 & -F & E & A \\ \end{matrix}\right)$$
Letting $w$ be a square root of $-1$ ($1$ if $p=2$) and computing the eigenvalues, we find that they are $(-w,-w,w,w).$
At this point we need to distinguish between $p=2$ and $p\neq 2$. 

Assume that $p\neq 2$.
We first assume that $F\neq 0$. We find that $Y_1$ is diagonalizable, i.e.
$$S_1^{-1}Y_1S_1= diag(-w,-w,w,w)$$
with $$S_1= \left(\begin{matrix}-C/F & (A-w)/F & -C/F & (A+w)/F \\ (A+w)/F & E/F & (A-w)/F & E/F \\ 0 & 1 & 0 & 1 \\ 1 & 0 & 1 & 0 \\ \end{matrix}\right).$$ Now observe that $$S_1^TJS_1= \left(\begin{matrix} 0 & \frac{2w}{F} & 0 & 0 \\ -\frac{2w}{F} & 0 & 0 & 0 \\ 0 & 0 & 0 & -\frac{2w}{F} \\ 0 & 0 & \frac{2w}{F} & 0 \\ \end{matrix}\right). $$
Multiply both sides by $\lambda_1 I$ for $\lambda_1 = \frac{wF}{2}$ and conjugate by the symmetric matrix $X:=\left(\begin{matrix} 0 & 0 & 1 & 0 \\ 0 & 1 & 0 & 0\\ 1 & 0 & 0 & 0 \\ 0 & 0 & 0 & 1\end{matrix}\right)$ to get $$(\lambda_1^{\frac{1}{2}}S_1X)^TJ(\lambda_1^{\frac{1}{2}}S_1X)=J.$$ This implies that $\lambda_1^{\frac{1}{2}}S_1X\in Sp_4$.   

Now if $F=0, E\neq 0$ we similarly find that $Y_1$ is diagonalizable, this time with $$S_1= \left(\begin{matrix}D/E & (A-w)/E & D/E & (A+w)/E \\ 0 & 1 & 0 & 1 \\ (-A-w)/E & 0 & (-A+w)/E & 0  \\ 1 & 0 & 1 & 0 \\ \end{matrix}\right)$$ and we can construct the same argument as for $F\neq 0$ to get again $$(\lambda_1^{\frac{1}{2}}S_1X)^TJ(\lambda_1^{\frac{1}{2}}S_1X)=J$$ so $\lambda_1^{\frac{1}{2}}S_1X\in Sp_4$.   
The same argument can be repeated in case at least one of $C,D,E,F$ is non-zero.

Now if $C=D=E=F=0$ we have $Y_1=\left(\begin{matrix}w & 0 & 0 & 0 \\ 0 & -w & 0 & 0 \\ 0 & 0 & -w & 0  \\ 0 & 0 & 0 & w \\ \end{matrix}\right)$ so that $XY_1X=diag(-w,-w,w,w)$ and $X^2=I \in Sp_4$.

Now if $M_2$ is another determinant $1$ singular matrix we proceed in the same way to obtain $\lambda_2^{\frac{1}{2}}S_2X\in Sp_4$ and $S_2^{-1}Y_2S_2= diag(-w,-w,w,w)$. To conclude simply note that we have $$(\lambda_1^{\frac{1}{2}}S_1)^{-1}Y_1(\lambda_1^{\frac{1}{2}}S_1)= diag(-w,-w,w,w)= (\lambda_2^{\frac{1}{2}}S_2)^{-1}Y_2(\lambda_2^{\frac{1}{2}}S_2) $$ and since we noted before that $\lambda_1^{\frac{1}{2}}S_1$ and $\lambda_2^{\frac{1}{2}}S_2$ lie in the same coset $Sp_4X^{-1}$ we have that $Y_1$ and $Y_2$ are conjugate by an element of $Sp_4$.

We now assume that $p=2$. In this case our matrix $Y_1$ has only the eigenvalue $w=1$. If one of $C,D,E,F$ is non-zero, using an analogous argument to the one before we find  $S_1,\lambda_1$ such that $$S_1^{-1}Y_1S_1= \left(\begin{matrix}1 & 1 & 0 & 0 \\ 0 & 1 & 0 & 0 \\ 0 & 0 & 1 & 1 \\ 0 & 0 & 0 & 1 \\ \end{matrix}\right)$$ with $\lambda_1 S_1\in Sp_4$ and if $Y_2$ is any other such matrix $Y_1$ and $Y_2$ are conjugate via an element  of $Sp_4$.

Finally if $C=D=E=F=0$ we have $Y_1=I$ and this is a fixed point of the conjugation action. This shows that if $p=2$ there are in fact two orbits on singular determinant $1$ matrices in $M_4$, while there is only $1$ otherwise. Since rank $4$ tensors correspond precisely matrices in $M_4$ of non-zero determinant and each $1$-space of rank $4$ contains a determinant $1$ matrix, $Sp_4\otimes Sp_4$ has $1$ orbit on rank $4$ singular $1$-spaces in $V=W_4\otimes W_4$ if $p\neq 2$, and $2$ orbits otherwise, as claimed.

For the stabilizer of a rank-$4$ singular $1$-space if $p\neq 2$, set $$M=diag(1,1,w,-w)$$ with $w$ as before the square root of $-1$. Now assume we have $(S_1,S_2)\in Sp_4\otimes Sp_4$ stabilizing the $1$-space $\langle M \rangle$. We must have $S_1MS_2^T=\lambda M$, where $\lambda$ is a fourth root of $1$, i.e. $\pm 1$ or $\pm w$. 
Set $Y=\tau(M^{-1})M$. Since $(SL_4)_\tau =Sp_4$, $S_1M^\tau S_2^T=\lambda M^\tau$ and therefore $Y=S_2^{-T}YS_2^T$. We can compute that $$Y=diag(w,-w,-w,w)$$ and this shows that $S_2$ is of the form
$$S_2=\left(\begin{matrix} a & 0 & 0 & b \\ 0 & c & d & 0 \\ 0 & e & f & 0 \\ g & 0 & 0 & h \\ \end{matrix}\right)$$ which is in $Sp_4$ if and only if $ah-bg=1$ and $cf-de=1$. The subgroup of $Sp_4$ generated by such elements is therefore isomorphic to $SL_2SL_2$.

We still require $S_1=\lambda M S_2^{-1} M^{-1} \in Sp_4$. A simple computation shows that $MS_2^{-T}M^{-1}\in Sp_4$ and therefore $S_1\in Sp_4$ if and only if $\lambda^2=1$. This proves that $(S_1,S_2)\in (Sp_4Sp_4)_{\langle M\rangle}$ if and only if $(S_1,S_2)\in (SL_2SL_2).2$.

If $p=2$ set $\langle M \rangle=\langle\left(\begin{matrix} 1 & 1 & 0 & 0 \\ 0 & 1 & 0 & 0 \\ 0 & 0 & 1 & 0 \\ 0 & 0 & 0 & 1 \\ \end{matrix}\right)\rangle$. We proceed as before, by finding $S_2\in Sp_4$ such that $S_2^{-1}YS_2=Y$ for $Y:=\tau(M^{-1})M=\left(\begin{matrix} 1 & 1 & 0 & 0 \\ 0 & 1 & 0 & 0 \\ 0 & 0 & 1 & 1 \\ 0 & 0 & 0 & 1 \\ \end{matrix}\right)$. We find that $S_2$ is of the form $$S_2=\left(\begin{matrix} a & b & c & d \\ 0 & a & 0 & c \\ e & f & g & h \\ 0 & e & 0 & g \\ \end{matrix}\right).$$ This happens if and only if $ag+ce=1$ and $S_2$ is contained in a $P_2$ parabolic (the stabilizer of $\langle e_1,f_2 \rangle$). The structure of $P_2$ is $U_3GL_2$ and since the subgroup of $P_2$ with $ag+ce=1$ is simply $U_3SL_2$, the stabilizer is as claimed.

We now consider the $Sp_4Sp_4$-orbits on rank-$3$ singular vectors. Fix a rank-$3$ tensor $v=v_1\otimes u_1+v_2\otimes u_2 +v_3\otimes u_3$. Since we know that $Sp_4$  acts transitively on $3$-spaces we can assume that $v=e_1\otimes u_1 + f_1\otimes u_2 + e_2\otimes u_3$. Note that singularity implies $(u_1,u_2)=0$. Since the $3$-space $\langle u_1,u_2,u_3 \rangle$ cannot be totally singular, we can assume without loss of generality that $(u_1,u_3) \neq 0$.  The stabilizer of  $\langle e_1,f_1\rangle$ in $Sp_4$ is $SL_2SL_2$ and since $(u_1,u_3)\neq 0$ we can find an element of $h\in SL_2\otimes 1$ such that $vh=e_1\otimes u_1' + f_1\otimes u_2' + e_2\otimes u_3'$, with $(u_2',u_3')=0$. We can therefore assume that $(u_2,u_3)=0$. By Witt's lemma  $Sp_4\otimes Sp_4$ acts transitively on the $1$-spaces with representatives of the form $v=e_1\otimes u_1 + f_1\otimes u_2 + e_2\otimes u_3$ with $(u_1,u_2)=(u_2,u_3)=0$ and $(u_1,u_3)\neq 0$. Therefore $Sp_4\otimes Sp_4$ acts transitively on singular $1$-spaces of rank $3$.

For the stabilizer of a representative of the orbit consider $(S_1,S_2)\in Sp_4\otimes Sp_4$ stabilising the singular $1$-space $\langle v\rangle = \langle e_1\otimes u_1 + f_1\otimes u_2 + e_2\otimes u_3 \rangle$, where we can assume that $u_1\in \langle u_1,u_2,u_3 \rangle^\perp$ and $(u_2,u_3)=1$, by previous discussion.
Since $S_1$ must stabilize $\langle e_1,f_1,e_2\rangle $ and $S_2$ must stabilize $\langle u_1,u_2,u_3\rangle$, we have that $S_1\leq (Sp_4)_{\langle e_2\rangle}$ and $S_2\in (Sp_4)_{\langle u_1 \rangle}$, i.e. $S_1$ and $S_2$ are contained in two $P_1$ parabolic subgroups. 
Denote by $\tilde{S_1}, \tilde{S_2}$ the $3\times 3$ matrices that represent the actions of $S_1,S_2$ respectively on $\langle e_1,f_1,e_2\rangle$ and $\langle u_1,u_2,u_3\rangle$ with respect to the given bases. We must have $\tilde{S_1}\tilde{S_2}^T=\alpha I$ for some $\alpha\in k$. 
We note that since $S_1$ stabilises $\langle e_2\rangle $, it must have form $$\tilde{S_1}=\left(\begin{matrix} a & b & 0 \\ c & d & 0  \\ e & f & g \\  \end{matrix}\right)$$ with the only condition being $ad-bc=1$. 

We now note that $\tilde{S_2}^T=\alpha\tilde{S_1}^{-1}$. A simple calculation shows that $\tilde{S_2}$ stabilizes $\langle u_1 \rangle$ and preserves the quadratic form if and only if $b=0$ and $\alpha^2 = dg$. Since the structure of a $P_1$ parabolic in $Sp_4$ is $U_3A_1T_1$ and we have the condition $b=0$ and $\alpha^2 = dg$, if $p\neq 2$ $S_1$ must be in $U_4T_1T_1.2$, with the extension by $2$ coming from the possibility of having $\alpha=\pm \sqrt{dg}$; if $p=2$ $S_1$ must be in $U_4T_1T_1$. We can still pick where $S_2$ sends a vector linearly independent from $ \langle u_1,u_2,u_3 \rangle$, which gives us an extra $U_1$.

Now consider a singular rank-$2$ tensor $a\otimes x + b\otimes y$. Singularity implies that either $(a,b)=0$ or $(x,y)=0$. Therefore either $(a,b)=(x,y)=0$, or $(a,b)=0,(x,y)\neq 0$ or $(a,b)\neq 0, (x,y)=0$. The three $Sp_4Sp_4$-orbits correspond to the three cases, by an application of Witt's Lemma.

We compute the stabilizers of the three $Sp_4Sp_4$-orbits. Consider $\langle a\otimes x + b\otimes y \rangle$ with $(a,b)=(x,y)=0$. The stabilizer of $\langle a,b \rangle $ in $Sp_4$ is a $P_2$ parabolic, with structure $U_3A_1T_1$. Let $S_1$ be an element of such $P_2$. If $(S_1,S_2)$ fixes $\langle a\otimes x + b\otimes y \rangle$. The element $S_2$ is in a parabolic subgroup $P_2'$ stabilising $\langle x,y\rangle$. Write $S_2=gu$ for $g\in A_1T_1$ and $u\in U_3$. The action of $S_2$ on $\langle x ,y\rangle$, i.e. $g$, is uniquely determined from $S_1$ up to a multiplicative constant, while $u$ is arbitrary. This shows that the full stabilizer of $\langle a\otimes x + b\otimes y \rangle$ is $U_6A_1T_2$.

Finally consider $\langle a\otimes x + b\otimes y \rangle$ with $(a,b)=0$ and $(x,y)\neq 0$. Let $(S_1,S_2)\in Sp_4Sp_4$ fixing $\langle a\otimes x + b\otimes y \rangle$. We require $S_1\in P_2$, where $P_2$ is a parabolic subgroup stabilizing $\langle a,b\rangle$. The stabilizer of $\langle x,y\rangle $ in $Sp_4$ has structure $A_1A_1$. The action of $S_2$ on $\langle x ,y\rangle$ is uniquely determined by $S_1$, while the action of $S_2$ on $\langle x,y\rangle ^\perp$ can be anything in $A_1$. We therefore have that the stabilizer of $\langle a\otimes x + b\otimes y \rangle$ has structure $P_2A_1=U_3A_1A_1T_1$.

Similarly it is easy to see that $Sp_4Sp_4$ has a single orbit on $1$-spaces of rank $1$, which are all singular, with stabilizer $P_1P_1$. 

This concludes our proof.
\end{proof}

We are finally able to conclude the proof of Theorem~\ref{main theorem 2}.

\begin{proof}
By  Lemma~\ref{two classical 4} Proposition~\ref{sp6sp6 main}, Lemma~\ref{sp4sp4} we only need to show that there are finitely many orbits on singular $1$-spaces in $V=V_4\otimes V_n$ when $H=H_1\otimes H_2$ is $Sp_4\otimes Sp_n$, with $n\geq 4$.

Any vector $v\in V_4\otimes V_n$ lies in a subspace of the form $V_4\otimes U$ with $\dim U\leq 4$. Since $H_2$ has finitely many orbits on subspaces, it suffices to show that $H_1\otimes S$ has finitely many orbits on singular $1$-spaces in $V_4\otimes U$, where $S$ is the group induced on $U$ by the stabilizer of $U$ in $H_2$.

If $U$ is contained in some non-degenerate $4$-space, then by Proposition~\ref{sp4sp4} $H_1\otimes S$ has finitely many orbits on singular $1$-spaces in $V_4\otimes U$. In particular this covers the cases $\dim U=1,2$.

If $\dim U=3$ or $\dim U =4$ and $U$ is totally singular, then we have $S=GL_3$ or $GL_4$ and we are in a finite orbit module case (see \cite[Table $1$]{finite} ).
The only case left to consider is $\dim U=4$ with $\dim Rad(U)=2$. Let $\{e_1,f_1,e_2,e_3\}$ be a basis for $U$, with $Rad(U)=\langle e_2, e_3\rangle $ and $(e_1,f_1)=1$. Now $H_1\otimes S$ has finitely many orbits on $1$-spaces in $(V_4\otimes U)/(V_4\otimes Rad(U))$, simply from Proposition~\ref{sp4sp4}. Any vector $v$ in $V_4\otimes U$ can be written as $v=v_1\otimes e_1+v_2\otimes f_1 +v_3\otimes e_2+v_4\otimes e_3$, for some $v_1,v_2,v_3,v_4\in V_4$. Note that we can assume that $v_1,v_2,v_3,v_4$ are linearly independent or $\langle v_2,v_2,v_3,v_4 \rangle$ is either totally singular or contained in some non-degenerate $4$-space, and we are back to one of the cases that we have already analyzed. It now suffices to show that $H_1\otimes S$ has finitely many orbits on $1$-spaces spanned by vectors of this form, with fixed $v_1\otimes e_1+v_2\otimes f_1$. This is true, since the group induced on $Rad(U)$ by $S$ is $GL_2$. 

Therefore $H_1\otimes S$ has finitely many orbits on singular $1$-spaces in $V_4\otimes U$ if $\dim Rad(U)=2$. 

For the stabilizers simply note that when $U$ is a non-degenerate $4$-space and $v$ is a rank-$4$ vector in $V_4\otimes U$, by Lemma~\ref{sp4sp4} the stabilizer of $\langle v \rangle$ in $Sp_4\otimes S$ is $A_1A_1.2$ when $p\neq 2$ and $U_3A_1$ when $p=2$. This gives stabilizers $(A_1A_1.2)Sp_{n-4}$ and $(U_3A_1)Sp_{n-4}$ in $Sp_4\otimes Sp_n$. By dimension considerations they must be generic stabilizers.

This concludes our proof.

\end{proof}

\bibliographystyle{acm}
\bibliography{biblio}{}

\end{document}